\documentclass[10pt]{amsart}
\usepackage{amssymb}

\usepackage{graphicx}
\usepackage{xcolor}

\newcommand{\res}{\upharpoonright}
\newcommand{\cu}{{\mathcal U}}
\newcommand{\cv}{{\mathcal V}}

\theoremstyle{plain}
\newtheorem{theorem}{Theorem}[section]
\newtheorem{corollary}[theorem]{Corollary}
\newtheorem{lemma}[theorem]{Lemma}
\newtheorem{proposition}[theorem]{Proposition}

\newtheorem*{theorem*}{Theorem}

\theoremstyle{remark}

\numberwithin{equation}{section}

\begin{document}

\bigskip

\title[Monoid actions and ultrafilter methods in Ramsey theory]{Monoid actions and\\ ultrafilter methods in Ramsey theory}

\author{S{\l}awomir Solecki}

\thanks{Research supported by NSF grant DMS-1266189 and DMS-1700426.}

\address{Department of Mathematics\\
Malott Hall\\
Cornell UNiversity\\
Ithaca, NY 14853, USA}

\email{ssolecki@cornell.edu}




\begin{abstract}
First, we prove a theorem on dynamics of actions of monoids by endomorphisms of semigroups. Second, 
we introduce algebraic structures suitable 
for formalizing infinitary Ramsey statements and prove a theorem that such statements are implied by the 
existence of appropriate homomorphisms between the algebraic structures. We make a connection between the two themes above, 
which allows us to prove some general Ramsey theorems for sequences. We give a new proof of the Furstenberg--Katznelson 
Ramsey theorem; in fact, we obtain a version of this theorem that is stronger than the original one. We answer in the negative 
a question of Lupini on possible extensions of Gowers' Ramsey theorem.      
\end{abstract}

\maketitle

\tableofcontents

\section{Introduction}
The main point of the paper is studying actions of monoids and establishing a relationship between monoid actions and Ramsey theory. Further, 
we make a connection with partial orders that have earlier proved important in set theoretic considerations. 

In Section~\ref{S:moa}, we study the dynamics of actions of monoids by continuous endomorphisms on compact right 
topological semigroups. We outline now the notions relevant to this study and its outcome: (1) a partial order ${\mathbb Y}(M)$; (2) a class of monoids; and (3) the theorem on dynamics of monoid actions. 

(1) We associate with each monoid $M$ a partial order ${\mathbb Y}(M)$ on which $M$ acts in an order preserving manner. 
We define first the order ${\mathbb X}(M)$ consisting of all principal right ideals in $M$, that is, sets of the form $aM$ for $a\in M$, 
with the order relation $\leq_{{\mathbb X}(M)}$ being inclusion. This order is considered in the representation theory of 
monoids as in \cite{St}. The monoid $M$ acts on ${\mathbb X}(M)$ by left translations. 
We then let ${\mathbb Y}(M)$ consist of all non-empty linearly ordered by $\leq_{{\mathbb X}(M)}$ 
subsets of ${\mathbb X}(M)$. We order ${\mathbb Y}(M)$ 
by end-extension, that is, we let $x\leq_{{\mathbb Y}(M)}y$ if $x$ is included in $y$ and all elements of $y\setminus x$ are larger 
with respect to $\leq_{{\mathbb X}(M)}$ than all elements of $x$. 
The construction of the partial order ${\mathbb Y}(M)$ from the partial order ${\mathbb X}(M)$ 
is a special case of a set theoretic construction going back to Kurepa \cite{Ku}. 
An order preserving action of the monoid $M$ on ${\mathbb Y}(M)$ 
is induced in the natural way from its action on ${\mathbb X}(M)$. 
   
(2) We introduce a class of monoids we call almost R-trivial, which contains the well known class of R-trivial monoids, see \cite{St}, and 
all the monoids of interest to us. In a monoid $M$, by the {\bf R-class} of $a\in M$ we understand, as in \cite{St}, the equivalence class of $a$ 
with respect to the equivalence relation that makes two elements equivalent if the principal right ideals generated by the two elements 
coincide, that is, $b_1$ and $b_2$ are equivalent if $b_1M=b_2M$. We call a monoid $M$  {\bf almost R-trivial} if for each 
element $b$ whose $R$-class has strictly more than one element we have $ab=b$ for each $a\in M$. (A monoid is R-trivial if the R-class of {\em each} element contains only that element.)  
In Section~\ref{Su:exa}, we provide the relevant examples of almost R-trivial monoids.   

(3) In Theorem~\ref{T:main}, which is the main theorem of Section~\ref{S:moa}, we show that each action of a finite almost R-trivial 
monoid by continuous endomorphisms on a compact right topological semigroup contains, in a precise sense, 
the action of $M$ on ${\mathbb Y}(M)$; in fact, it contains a natural action of $M$ by endomorphisms on the semigroup $\langle {\mathbb Y}(M)\rangle$ generated canonically by ${\mathbb Y}(M)$, Corollary~\ref{C:rest}. 
This result was inspired by Ramsey theoretic considerations, 
but it may also be of independent dynamical interest. 

In Section~\ref{S:str}, we introduce new algebraic structures, we call {\bf function arrays}, that are appropriate for formalizing various Ramsey statements 
concerning sequences. We isolate the notions of {\bf basic sequence} and {\bf tame coloring}. In Theorem~\ref{T:abst}, the main 
theorem of this section, we show that 
finding a basic sequence on which a given coloring is tame follows from the existence of an appropriate homomorphism. 
This theorem reduces proving a Ramsey statement to establishing an algebraic property. We introduce a natural notion 
of tensor product of the algebraic structures studied in this section, which makes it possible to strengthen the conclusion 
of Theorem~\ref{T:abst}. 

In Section~\ref{S:conn}, we connect the previous two sections with each other and explore Ramsey theoretic 
issues. In Corollary~\ref{C:yet}, we show that the main result of 
Section~\ref{S:moa} yields a homomorphism required for the main result of Section~\ref{S:str}. This  corollary has various 
Ramsey theoretic consequences. For example, 
we introduce a notion of {\bf Ramsey monoid} and prove that, among finite almost R-trivial monoids $M$, 
being Ramsey is equivalent to linearity of the order ${\mathbb X}(M)$.
We use this result to show that an extension of Gowers' Ramsey theorem \cite{Go} inquired for by Lupini \cite{Lu} is 
false. As other consequences, we obtain some earlier Ramsey results by associating with each of them 
a finite almost R-trivial monoid. For example, we  
show the Furstenberg--Katznelson Ramsey theorem for located words, which is stronger than the original version of the 
theorem from \cite{FK}. Our proof is also different from the one in \cite{FK}. 

We state here one Ramsey theoretic result from Section~\ref{S:conn}, 
which has Furs\-ten\-berg--Katznelson's and Gowers' theorems, \cite {FK}, \cite{Go}, as special instances; see Section~\ref{Su:conc}. 
Let $M$ be a monoid. 
By a {\bf located word over} $M$ we understand a function from a finite non-empty subset of $\mathbb N$ to $M$. For two such words 
$w_1$ and $w_2$, we write $w_1\prec w_2$ if the largest element of the domain of $w_1$ is smaller than the smallest element of the 
domain of $w_2$. In such a case, we write $w_1w_2$ for the located word that is the function whose graph is the union of the graphs 
of $w_1$ and $w_2$. For a located word $w$ and $a\in M$, we write $a(w)$ for the located word that results from multiplying on 
the left each value of $w$ by $a$. Given a finite coloring of all located words, we are interested in producing a 
sequence $w_0\prec w_1\prec \cdots$ of located words, for which we control the color of 
\[
a_0(w_{n_0})\cdots a_k(w_{n_k}),  
\]
for arbitrary $a_0, \dots, a_k\in M$ and $n_0<\cdots <n_k$. 
The control over the color is exerted using the partial order ${\mathbb Y}(M)$ introduced above. 
With each partial order $(P, \leq_P)$, one naturally associates a semigroup $\langle P\rangle$, with 
its binary operation denoted by $\vee$, that is the semigroup generated freely by the elements of $P$ subject to the relations 
\begin{equation}\label{E:pq}
p\vee q = q\vee p = q\;\hbox{ if } p\leq_Pq. 
\end{equation}
We consider the semigroup $\langle {\mathbb Y}(M)\rangle$ produced from the partial order ${\mathbb Y}(M)$ in this manner. 
We now have the following statement, which is proved as Corollary~\ref{C:genseq}. 

\smallskip
\noindent {\em Let $M$ be almost R-trivial and finite. 
Fix a finite subset $F$ of the semigroup 
$\langle {\mathbb Y}(M)\rangle$ and a maximal element $\bf y$ of the partial order ${\mathbb Y}(M)$. 
For each coloring with finitely many colors of all located words over $M$, there exists a sequence  
\[
w_0\prec w_1\prec w_2\prec \cdots
\]
of located words such that the color of 
\[
a_0(w_{n_0})\cdots a_k(w_{n_k}),  
\]
for $a_0, \dots, a_k\in M$ and $n_0<\cdots <n_k$, depends only on the element 
\[
a_0({\bf y})\vee \cdots \vee a_k({\bf y})
\]
of $\langle {\mathbb Y}(M)\rangle$ provided that $a_0({\bf y})\vee \cdots \vee a_k({\bf y})\in F$. }
\smallskip

One can view $a_0({\bf y})\vee \cdots \vee a_k({\bf y})$ as the ``type" of $a_0(w_{n_0})\cdots a_k(w_{n_k})$ and the theorem as asserting that the color of $a_0(w_{n_0})\cdots a_k(w_{n_k})$ depends only on its type. 
In general, 
the element $a_0({\bf y})\vee \cdots \vee a_k({\bf y})$ contains much less information than the located word 
$a_0(w_{n_0})\cdots a_k(w_{n_k})$, 
due partly to the disappearance of $w_{n_0}, \dots, w_{n_k}$ and partly to the influence of relations \eqref{E:pq}.

We comment now on our view of the place of the present work within Ramsey theory. 
A large portion of Ramsey Theory can be parametrized by a triple $(a,b,c)$, where $a, b, c$ are natural numbers or $\infty$ and 
$a\leq b\leq c$. (We exclude here, for example, a very important part of Ramsey theory called structural Ramsey theory, 
for which a general approach is advanced in \cite{HN}.) 
The simplest Ramsey theorems are those associated with $a\leq b<\infty=c$. (For example, for each finite coloring of all 
$a$-element subsets
of an infinite set $C$, there exists a $b$-element subset of $C$ such that all of its $a$-element subsets get the same color.)
These simplest theorems are strengthened in two directions.

Direction 1: $a\leq b\leq c<\infty$. This is the domain of Finite Ramsey Theory. (For example, for each finite coloring of all 
$a$-element subsets
of a $c$-element set $C$, there exists a $b$-element subset of $C$ such that all of its $a$-element subsets get the same color.) 
Appropriate structures for this part of the theory are described in \cite{So}.

Direction 2: $a=b=c=\infty$. This is the domain of Infinite Dimensional Ramsey Theory. 
(For example, for each finite Borel coloring of all infinite element subsets
of an infinite countable set $C$, there exists an infinite subset of $C$ such that all of its infinite subsets get the same color.) 
Appropriate structures for this theory were developed in \cite{To}
and a General Ramsey Theorem for them was proved there.

The frameworks in 1 and 2 are quite different in particulars, but, roughly speaking, 
the General Ramsey Theorems (GRT) in both cases have the same form:

\centerline{GRT: Pigeonhole Principle implies Ramsey Statement.}

Such GRT, reduces proving concrete Ramsey statements to proving appropriate pigeonhole principles.
In 1, pigeonhole principles are either easy to check directly or, more frequently, they are reformulations of Ramsey 
statements proved earlier
using GRT with the aid of easier pigeonhole principles. So it is a self-propelling system. In 2, pigeonhole principles cannot be obtained 
this way and they require separate proofs.
(The vague reason for this is that the pigeonhole principles here correspond to the case $b=c=\infty$ and $a = \hbox{potential }\infty$.)

This paper can be viewed as providing appropriate structures and general theorems that handle proofs of pigeonhole 
principles in 2. These structures are quite different from those in 1 and 2.

The concurrently written interesting paper \cite{Lu2} also touches on the theme of ultrafilter methods in Ramsey theory. 
This work and ours are independent from each other.

\section{Monoid actions on semigroups}\label{S:moa}

The theme of this section is purely dynamical. We study actions of finite monoids on compact right topological semigroups 
by continuous endomorphisms. We isolate the class of almost R-trivial monoids that extends the well studied class of R-trivial monoids. 
We prove in Theorem~\ref{T:main} 
that each action of an almost R-trivial finite monoid $M$ on a compact right topological semigroup by continuous endomorphisms 
contains, in a precise sense, a finite action defined only in terms of $M$. This finite action is an action of $M$ on a partial order 
${\mathbb Y}(M)$ introduced in Section~\ref{Su:Mpa}. An important to us reformulation of Theorem~\ref{T:main} is done 
in Corollary~\ref{C:rest}.

\subsection{Monoid actions on partial orders}\label{Su:Mpa}

A {\bf monoid} is a semigroup with a distinguished element that is a left and right identity. 
By convention, if a monoid acts on a set, the identity element acts as the identity function.

Let $M$ be a monoid.
By an {\bf $M$-partial order} we understand a set $X$ equipped with an action of $M$ and with a partial order $\leq_X$ such that
if $x\leq_X y$, then $ax\leq_X ay$, for $x,y\in X$ and $a\in M$.
Let $X$ and $Y$ be $M$-partial orders. A function $f\colon X\to Y$ is an {\bf epimorphism}
if $f$ is onto, $f$ is $M$-equivariant, and $\leq_Y$ is the image under $f$ of $\leq_X$.
We say that an $M$-partial order $X$ is {\bf strong} if, for all $y\in X$ and $a\in M$,
\[
\{ ax\in X\colon x\leq_X y\} = \{ x\in X\colon x\leq_X ay\}.
\]

For a monoid $M$, consider $M$ acting on itself by multiplication on the left. Set
\begin{equation}\label{E:XM}
{\mathbb X}(M)= \{ aM\colon a\in M\}
\end{equation}
with the order relation being inclusion. Then, ${\mathbb X}(M)$ is an $M$-partial order. We actually have more.

\begin{lemma}
Let $M$ be a monoid. Then ${\mathbb X}(M)$ is a strong $M$-partial order.
\end{lemma}

\begin{proof} We need to see that if $cM\subseteq abM$, then there is $c'$ such that $c'M\subseteq bM$ and
$ac'M = cM$. Since $cM\subseteq abM$, we have $c\in abM$, so $c = abd$ for some $d\in M$. Let $c'=bd$. It is easy to check that
this $c'$ works.
\end{proof}

For each finite partial order $X$, let
\begin{equation}\label{E:fr} 
{\rm Fr}(X) = \{ x\subseteq X\colon x\not=\emptyset\hbox{ and } x\hbox{ is linearly ordered by }\leq_X\}.
\end{equation}
The order relation on ${\rm Fr}(X)$ is defined by letting for $x,y\in {\rm Fr}(X)$,
\[
x\leq_{{\rm Fr}(X)} y \Longleftrightarrow x\subseteq y\hbox{ and } i <_X j \hbox{ for all }i\in x\hbox{ and }j\in y\setminus x.
\]
Observe that ${\rm Fr}(X)$ is a {\bf forest}, that is, it is a partial order in which the set of predecessors of each element is linearly ordered. 
(We take this sentence as our definition of the notion of forest.)
As pointed out by Todorcevic, the operation $\rm Fr$ is a finite version of certain constructions from infinite combinatorics of
partial orders \cite{Ku}, \cite{To-1}.

Let $X$ be an $M$-partial order. For $x\in {\rm Fr}(X)$ and $a\in M$, let
\[
ax = \{ ai\colon i\in x\}.
\]
Clearly, $ax\in {\rm Fr}(X)$ and $M\times {\rm Fr}(X)\ni (a,x)\to ax \in {\rm Fr}(X)$ is an action of $M$ on ${\rm Fr}(X)$.

The following lemma is easy to verify.

\begin{lemma}\label{L:for}
Let $M$ be a monoid, and let $X$ be a finite $M$-partial order.
\begin{enumerate}
\item[(i)] ${\rm Fr}(X)$ with the action defined above is a strong $M$-partial order.

\item[(ii)] The function $\pi\colon {\rm Fr}(X)\to X$ given by $\pi(x) = \max x$ is an epimorphism between the two $M$-partial orders.
\end{enumerate}
\end{lemma}

For a finite monoid $M$, set
\begin{equation}\label{E:YM}
{\mathbb Y}(M) = {\rm Fr}({\mathbb X}(M)).
\end{equation}
By Lemma~\ref{L:for}, ${\mathbb Y}(M)$ is a strong $M$-partial order.

\subsection{Compact right topological semigroups}

We recall here some basic notions concerning right topological semigroups. 

Let $U$ be a semigroup.
As usual, let
\[
E(U)
\]
be the set of all idempotents of $U$.
There is a natural transitive, anti-symmetric relation $\leq^U$ on $U$ defined by
\[
u\leq^U v\Longleftrightarrow uv=vu=u.
\]
This relation is reflexive on the set $E(U)$. So $\leq^U$ is a partial order on $E(U)$.

A semigroup equipped with a topology is called {\bf right topological} if, for each $u\in U$, the function
\[
U\ni x\to xu\in U
\]
is continuous.

In the proposition below, we collect facts about idempotents in compact semigroups needed here. They are
proved in \cite[Lemma 2.1, Lemma~2.3 and Corollary~2.4, Lemma 2.11]{To}.

\begin{proposition}\label{P:fa}
Let $U$ be a compact right topological semigroup.
\begin{enumerate}
\item[(i)] $E(U)$ is non-empty.

\item[(ii)] For each $v\in E(U)$ there exists a minimal with respect to $\leq^U$ element $u\in E(U)$ with $u\leq^U v$.

\item[(iii)] For each minimal with respect to $\leq^U$ element $u\in E(U)$ and each right ideal $J\subseteq U$, there exists
$v\in J\cap E(U)$ with $uv=u$.
\end{enumerate}
\end{proposition}

If $U$ is equipped with a compact topology, that may not interact with multiplication in any way, then there exists
the smallest under inclusion compact two-sided ideal of $U$, see \cite{HS}. So, for a compact right topological semigroup $U$, let
\[
I(U)
\]
stand for the smallest compact two-sided ideal with respect to the compact topology on $U$.

\subsection{Almost R-trivial monoids}\label{Su:exa}

Two elements $a,b$ of a monoid $M$ are called {\bf R-equivalent} if $aM=bM$. Of course, by an {\bf R-class} of $a\in M$ we understand
the set of all elements of $M$ that are R-equivalent to $a$. A monoid $M$ is called {\bf R-trivial} if each R-class has exactly one element,
that is, if for all $a,b\in M$, $aM=bM$ implies $a=b$. This notion with an equivalent definition was introduced in \cite{Sc}. For the role
of R-trivial monoids in the representation theory of monoids see \cite[Chapter 2]{St}.

Note that if $M$ is R-trivial, then the partial order ${\mathbb X}(M)$
can be identified with $M$ taken with the partial order $a\leq_M b$ if and only if $a\in bM$.
We call a monoid $M$ {\bf almost R-trivial} if, for each $b\in M$ whose R-class has more than one element, we have $ab=b$
for all $a\in M$.

We present now examples of almost R-trivial monoids relevant in Ramsey theory.

{\bf Examples.} {\bf 1.} Let $n\in {\mathbb N}$, $n>0$.
Let
\[
G_n
\]
be $\{ 0, \dots, n-1\}$ with multiplication defined by
\[
i\cdot j = \min (i+j, n-1).
\]
We set $1_{G_n} =0$.

The monoid $G_n$ is R-trivial since, for each $i\in G_n$, we have $iG_n = \{ i, \dots, n-1\}$.

The monoid $G_n$ is associated with Gowers' Ramsey theorem \cite{Go}, see also \cite{To}.

{\bf 2.} Fix $n\in {\mathbb N}$, $n>0$. Let
\[
I_n
\]
be the set of all non-decreasing functions
that map $n$ onto some $k\leq n$. These are precisely the non-decreasing functions $f\colon n\to n$ such that $f(0)=0$ and
$f(i+1)\leq f(i)+1$ for all $i<n-1$. The multiplication operation is composition and $1$ is the identity function from $n$ to $n$.

The monoid $I_n$ is R-trivial. To see this,
let $f, g\in I_n$ be such that $f\in gI_n$ and $g\in fI_n$, that is, $f=g\circ h_1$ and $g=f\circ h_2$, for some $h_1, h_2\in I_n$.
It follows from these equations that $f(i)\leq g(i)$, for all $1\leq i\leq n$, and $g(i)\leq f(i)$, for all $1\leq i\leq n$. Thus, $f=g$.

The monoid $I_n$ is associated with Lupini's Ramsey theorem \cite{Lu}.

{\bf 3.} Fix two disjoint sets $A, B$, and let $1$ not be an element of $A\cup B$.
Let
\[
J(A,B)
\]
be $\{ 1\} \cup A\cup B$. Define multiplication on $J(A,B)$ by letting, for each $c\in A\cup B$,
\[
\begin{split}
&c\cdot a = c, \hbox{ if }a\in A;\\
&c\cdot b = b,\hbox{ if } b\in B.
\end{split}
\]
Of course, we define $1 \cdot c = c\cdot 1 = c$ for all $c\in J(A,B)$. We leave it to the reader to check that so defined multiplication
is associative.

The monoid $J(A,B)$ is almost R-trivial. Indeed, a quick check gives, for $a\in A$ and $b\in B$,
\[
aJ(A,B) = \{ a\}\cup B,\; bJ(A,B) = B,\; 1J(A,B)= J(A,B).
\]
Thus, the only elements of $J(A,B)$, whose R-classes can
possible have size bigger than one, are elements of $B$. But for all $c\in J(A,B)$ and $b\in B$, we have $cb=b$. It follows that
$J(A,B)$ is almost R-trivial (and not R-trivial if the cardinality of $B$ is strictly bigger than one).

The monoid $J(\emptyset, B)$ for a one element set $B$ is associated with Hindman's theorem, see \cite{To}, and for arbitrary finite $B$
with the infinitary Hales--Jewett theorem, see \cite{To}.
For arbitrary finite $A$ and $B$, $J(A,B)$ is associated with the Furstenberg--Katznelson theorem \cite{FK}.

\subsection{The theorem on monoid actions}\label{Su:theomain}

\indent {\em In the results of this section, we adopt the following conventions:
\begin{enumerate}
\item[---] $U$ is a compact right topological semigroup;

\item [---] $M$ is a finite monoid acting on $U$ by continuous endomorphisms.
\end{enumerate}}

The following theorem is the main result of this section. 

\begin{theorem}\label{T:main}
Assume $M$ is almost R-trivial. There exists a function $g\colon {\mathbb Y}(M)\to E(U)$ such that
\begin{enumerate}
\item[(i)] $g$ is $M$-equivariant;

\item[(ii)] $g$ is order reversing with respect to $\leq_{{\mathbb Y}(M)}$ and $\leq^U$;

\item[(iii)] $g$ maps maximal elements of ${\mathbb Y}(M)$ to $I(U)$.
\end{enumerate}

Moreover, if ${\mathbb X}(M)$ has at most two elements, then $g$ maps maximal elements of ${\mathbb Y}(M)$ to minimal 
elements of $E(U)$. 
\end{theorem}

We will need the following lemma.

\begin{lemma}\label{L:aux}
Let $F$ be a strong $M$-partial order that is a forest. Assume that $f\colon F\to U$ is $M$-equivariant.
Then there exists $g\colon F\to E(U)$ such that
\begin{enumerate}
\item[(i)] $g$ is $M$-equivariant;

\item[(ii)] $g$ is order reversing with respect to $\leq_F$ and $\leq^U$;

\item[(iii)] $g^{-1}(I(U))$ contains $f^{-1}(I(U))$.
\end{enumerate}
\end{lemma}

\begin{proof} Let $A\subseteq F$ be downward closed. Assume we have a function $g_A\colon F\to E(U)$ such that (i) and (iii) hold and
additionally, for all $i, j\in A$,
\begin{enumerate}
\item[($\star$)] if  $i <_F j$, then $g_A(j) g_A(i)= g_A(j)$.
\end{enumerate}
Note the condition that the values of $g_A$ are in $E(U)$, so they are idempotents.
Let $B\subseteq F$ be such that $A\subseteq B$ and all the immediate predecessors of elements of $B$ are in $A$. We claim that there
exists $g_B\colon F\to E(U)$ fulfilling (i), (iii),  ($\star$) for all $i, j\in B$, and $g_B\res A = g_A\res A$.

First, define $g'_B\colon F\to U$ by letting, for $j\in F$,
\[
g'_B(j) = g_A(i_k) g_A(i_{k-1}) \cdots g_A(i_1),
\]
where $i_1<_F \cdots <_F i_k$ lists the set $\{ i\in F\colon i\leq_F j\}$ is the increasing order.

We check that $g'_B$ fulfills (i), (iii), and ($\star$) for $i, j\in B$.
Point (i) holds since for each $a\in M$ we have
\begin{equation}\label{E:iii}
\begin{split}
a(g'_B(j)) &= a(g_A(i_k)) a(g_A(i_{k-1}))\cdots a(g_A(i_1))\\
&= g_A(a(i_k)) g_A(a(i_{k-1}))\cdots g_A(a(i_1)) = g'_B(a(j)),
\end{split}
\end{equation}
with the second equality holding since the function $g_A$ is $M$-equivariant and
the third one holding by idempotence of the values of $g_A$ and the fact that $F$ is an $M$-forest. Point (iii) holds
since $I(U)$ is a right ideal and the function $g_A$ fulfills (iii). To check ($\star$) for
$i, j\in B$ with $i<_Fj$, let
\[
i_1<_F \cdots <_F i_k<_F \cdots <_F i_l
\]
list all the predecessors of $j$ in the increasing order so that $i_k=i$ and, of course, $i_l=j$. Then, since $j\in B$,
we have $i_1, \dots, i_{k}\in A$ and, therefore, we get
\[
g'_B(j) = g_A(i_l)\cdots g_A(i_k) = g_A(j)\cdots g_A(i).
\]
By the same computation carried out for $i=j\in A$, we see
\begin{equation}\label{E:row}
g'_B(i) = g_A(i), \hbox{ for }i\in A.
\end{equation}
It follows that
\[
g'_B(j) g'_B(i)= g_A(j)\cdots g_A(i) g_A(i) = g_A(j)\cdots g_A(i) = g'_B(j).
\]
This equality shows that ($\star$) holds for $i, j\in B$. Finally, note that \eqref{E:row} implies that $g'_B\res A = g_A\res A$. Thus, $g'_B$
has all the desired properties.

To construct $g_B$ from $g'_B$,
consider $U^F$ with coordinstewise multiplication and the product topology. This is a right topological semigroup. Define
$H\subseteq U^F$ to consist of all $x\in U^F$ such that
\begin{enumerate}
\item[($\alpha$)] the function $F\ni i\to x_i\in U$ fulfills (i), (iii), and ($\star$) for $i, j\in B$ and

\item[($\beta$)] $x_i = g_A(i)$ for all $i\in A$.
\end{enumerate}

First we observe that $H$ is a subsemigroup of $U^F$. Condition (i) is clearly closed under multiplication. Condition (iii) is closed under multiplication since
$I(U)$ is a two-sided ideal. Condition ($\star$) is closed under multiplication in the presence of ($\beta$) since, for $x,y\in H$ and $i,j\in B$ with $i<_Fj$,
we have $i\in A$ and, therefore,
\[
x_jy_jx_iy_i = x_jy_j g_A(i)g_A(i) = x_jy_j y_iy_i = x_jy_j.
\]
This verification shows that ($\alpha$) is closed under multiplication in the presence of ($\beta$).
Condition ($\beta$) is closed under multiplication since $g_A(i)$ is an idempotent.

Next note that $H$ is compact since all conditions defining $H$ are clearly topologically closed with a possible exception of ($\star$) for $i,j\in B$
with $i<_Fj$. Note that in this case $i\in A$. Since $x\in U^F$ and $i\in A$,
we have $x_i=g_A(i)$, condition ($\star$) translates to $x_j g_A(i) = x_j$ for $i\in A$ and $j\in B$ with $i<_Fj$. This condition
is closed since $U$ is right topological. Finally note that $H$ is non-empty since $g'_B$ is its element.
By Ellis' Lemma~\cite[Lemma 2.1]{To}, $H$ contains an idempotent. Let $g_B\in H$ be such an idempotent. It has all the required properties.

The above procedure describes the passage from $g_A$ to $g_B$ if all immediate predecessors of elements of $B$ are in $A$.

We now define $g_\emptyset\colon F\to E(U)$ that fulfills (i), (iii), and ($\star$) for $A=\emptyset$, with the last condition holding vacuously. Note that $f$ has 
all the properties required of $g_\emptyset$ 
except its values may not be in $E(U)$. To remedy this shortcoming, consider again the compact right topological semigroup $U^F$ with coordinstewise
multiplication and the product topology. Define $H\subseteq U^F$ to consist of all $x\in U^F$ such that
the function $F\ni i\to x_i\in U$ fulfills (i) and (iii) (and, vacuously, ($\star$) for $i, j\in \emptyset$). Then $H$ is non-empty since $f\in H$. As above, 
we check that $H$ is a compact subsemigroup of $U^F$. Let
$g_\emptyset$ be an idempotent in $H$. Clearly, $g_\emptyset$ has the required properties. 

Starting with $g_\emptyset$ and recursively using the procedure of going from $g_A$ to $g_B$ described above, we produce $g_F\colon F\to E(U)$
fulfilling (i), (iii) and ($\star$) for all $i,j\in F$. 

Now define, for $j\in F$, 
\begin{equation}\label{E:fig}
g(j) = g_F(j_1)\cdots g_F(j_l),
\end{equation}
where $j_1<_F \cdots <_F j_l = j$ list all elements of the set $\{ i\in F\colon i\leq_F j\}$. This $g$ is as required by the conclusion of the lemma. 
Keeping in mind that all values of $g_F$ are idempotents, we see that point (i) for $g$ holds by the calculation as in \eqref{E:iii}. Point (iii) for $g$ is clear since 
it holds for $g_F$, $j_l=j$ in formula \eqref{E:fig}, and $I(U)$ is a two-sided, so left, ideal. To see point (ii) for $g$, 
we use an argument similar to one applied earlier in the proof. 
To do this, fix $i\leq_F j$ in $F$, and let 
\[
i_1<_F\cdots <_F i_k = i\;\hbox{ and }\; j_1<_F \cdots <_F j_l=j
\]   
list elements of $F$ that are $\leq_F i$ and $\leq_F j$, respectively. Note that $i_1 = j_1, \dots, i_k=j_k$.
Using ($\star$) for $g_F$ and idempotency of $g_F(i_k) = g_F(j_k)$, we see that 
\[
\begin{split}
g(i)g(j) &= g_F(i_1)\cdots g_F(i_k) g_F(j_1)\cdots g_F(j_l)\\ 
&= g_F(i_1)\cdots g_F(i_k) g_F(j_k)\cdots g_F(j_l) = g(j), 
\end{split}
\]
while using only ($\star$) for $g_F$ if $i<_Fj$ and ($\star$) and idempotency of $g_F(j)$ if $i=j$, we get
\[
\begin{split}
g(j)g(i) &= g_F(j_1)\cdots g_F(j_l) g_F(i_1)\cdots g_F(i_k)\\ 
&= g_F(j_1)\cdots g_F(j_l)  = g(j). 
\end{split}
\]
Thus, point (ii) is verified for $g$. Note that point (ii) implies that values of $g$ are idempotent, that is, they are elements of $E(U)$. Therefore, we checked that 
$g\colon F\to E(U)$ and (i)--(iii) hold for $g$. 
\end{proof}

\begin{lemma}\label{L:sq}
Assume that $ab=b$, for all $a,b\in M$ with $b\not= 1_M$. Then there exists minimal $u_1\in E(U)$ such that
\begin{enumerate}
\item[(i)] $u_1\in I(U)$

\item[(ii)] $a(u_1) = b(u_1)$, for all $a, b\in M$ with $a\not= 1_M \not= b$.
\end{enumerate}
\end{lemma}

\begin{proof} Observe that, for $a, b\in M\setminus \{ 1_M\} $, since $ba=a$, we have
\[
a(U) = ba(U)= b(a(U))\subseteq b(U).
\]
By symmetry, we see that $a(U)=b(U)$. Let $T$
be the common value of the images of $U$ under the elements of $M\setminus \{ 1_M\} $.
Clearly $T$ is a compact subsemigroup of $U$. Note that
\begin{equation}\label{E:para}
a(u)=u,\,\hbox{ for }u\in T,\, a\in M.
\end{equation}
Let
\[
u_0\in T
\]
be a minimal with respect to $\leq^T$ idempotent.

Let $u_1\in U$ be a minimal idempotent in $U$ with $u_1\leq^U u_0$. Since $u_1$ is minimal, we have
\begin{equation}\label{E:a}
u_1\in I(U).
\end{equation}
We show that
\begin{equation}\label{E:au}
a(u_1) = u_0,\hbox{ for all } a\in M\setminus \{ 1_M\} .
\end{equation}
Indeed, since $u_1\leq^{U} u_0$ and $u_0\in T$, by \eqref{E:para}, we get
\[
a(u_1)\leq^{U} a(u_0) = u_0.
\]
Thus, $a(u_1)\leq^T u_0$ and $a(u_1)\in T$. Since $u_0$ is minimal in $T$, we get $a(u_1)=u_0$.

Equations \eqref{E:a} and \eqref{E:au} show that $u_1$ is as required.
\end{proof}

\begin{proof}[Proof of Theorem~\ref{T:main}]
Let
\[
B = \{ b\in M\colon ab=b \hbox{ for all }a\in M\}.
\]
Note that $M'= \{ 1_M\}\cup B$ is a monoid fulfilling the assumption of
Lemma~\ref{L:sq}. Let $u_1\in U$ be an element as in the conclusion of Lemma~\ref{L:sq}.

Define a function $h\colon M\to U$ by $h(a)=a(u_1)$. Note that $h$ is $M$-equivariant if $M$ is taken with left multiplication action.
Observe the following two implications:
\begin{enumerate}
\item[(a)] if $a_1\in M\setminus B$, $a_2\in M$, and $a_1M=a_2M$, then $a_1=a_2$;

\item[(b)] if $a\in M$ and $b\in B$, then $bM\subseteq aM$.
\end{enumerate}
Point (a) follows from $M$ being almost R-trivial. Point (b)
is a consequence of $b=ab\in aM$. Let $\rho\colon M\to {\mathbb X}(M)$ be the equivariant surjection $\rho(a) = aM$. Note that
by (a) and (b), $\rho$ is injective on $M\setminus B$, all points in $B$ are mapped to a single point of ${\mathbb X}(M)$ that
is the smallest point of this partial order, and no point of $M\setminus B$ is mapped to this smallest point. 
It now follows from the properties of $u_1$ listed in Lemma~\ref{L:sq}
that $h$ factors through $\rho$ giving a function $h'\colon {\mathbb X}(M)\to U$ with $h'\circ \rho = h$. 
Since $\rho$ and $h$ are $M$-equivariant, so is $h'$. 
Let $\pi\colon {\mathbb Y}(M)\to {\mathbb X}(M)$ be the $M$-equivariant function given by Lemma~\ref{L:for}(ii).
Then $f\colon {\mathbb Y}(M)\to U$, defined by $f = h'\circ \pi$, is $M$-equivariant. Furthermore, since $u_1\in I(U)$ gives
$h(1_M)\in I(U)$, and hence $h'(1_MM)\in I(U)$, 
we see that the maximal elements of ${\mathbb Y}(M)$ are mapped by $f$ to $I(U)$.  
Note that if ${\mathbb X}(M)$ has at most two elements, then we can let $g = f$. Then $h'(1_MM)$ 
is an idempotent minimal with respect to $\leq^U$, and $g$ maps all maximal elements of 
${\mathbb Y}(M)$ to $h'(1_MM)$. 
Without any restrictions on the size of ${\mathbb X}(M)$, Lemma~\ref{L:aux} can be applied to $f$ giving a function $g$ as required 
by points (i)--(iii).
\end{proof}

\subsection{Semigroups from partial orders and a restatement of the theorem}\label{Su:sempar}

For a partial order $P$, let 
\[
\langle P\rangle 
\]
be the semigroup, whose binary operation is denoted by $\vee$, generated freely by elements of $P$ modulo the relations 
\begin{equation}\label{E:rel}
p\vee q = q\vee p = q, \hbox{ for }p,q\in P\hbox { with }p\leq_Pq.
\end{equation}
That is, each element of $\langle P\rangle$ can be uniquely written as $p_0\vee \cdots \vee p_n$ for some $n\in {\mathbb N}$ and 
with $p_i$ and $p_{i+1}$ being incomparable with respect to $\leq_P$, for all $0\leq i<n$. Note that if $P$ is linear, then 
$\langle P\rangle = P$.

Observe that if $M$ is a monoid and $P$ is an $M$-partial order, then the action of $M$ on $P$ naturally induces an action of $M$ on 
$\langle P\rangle$ by endomorphisms, namely, for $a\in M$ and $p_0\vee\cdots \vee p_n\in \langle P\rangle$ with $p_0, \dots , p_n\in P$, we let 
\begin{equation}\label{E:orsem}
a(p_0\vee\cdots \vee p_n) = a(p_0)\vee\cdots \vee a(p_n).
\end{equation}  
It is easy to see that the right hand side of the above equality is well defined and that formula \eqref{E:orsem} defines an endomorphism of $\langle P\rangle$ and, in fact,
an action of $M$ on $\langle P\rangle$. 

A moment of thought convinces one that the function from Theorem~\ref{T:main} extends to a homomorphism 
from $\langle {\mathbb Y}(M)\rangle$ to $U$---condition (ii) of Theorem~\ref{T:main} and the fact that the function in that 
theorem takes values in $E(U)$ are responsible for this. Therefore, we get the following corollary, which we state with the conventions 
of Section~\ref{Su:theomain}. 

\begin{corollary}\label{C:rest} 
Assume $M$ is almost R-trivial. There exists an $M$-equivariant homomorphism of semigroups 
$g\colon \langle {\mathbb Y}(M)\rangle \to U$ that maps maximal elements of ${\mathbb Y}(M)$ to $I(U)$. Additionally, if ${\mathbb X}(M)$ has at most two elements, then 
$g$ maps maximal elements of ${\mathbb Y}(M)$ to minimal idempotents in $U$.
\end{corollary}

\section{Infinitary Ramsey theorems}\label{S:str}

The goal of this section is Ramsey theoretic. We introduce structures, we call function arrays, 
that generalize the partial semigroup setting of \cite{BBH}. One important feature of these structures is their closure under naturally defined tensor product. 
For function arrays, we introduce the notion of basic sequence. 
Basic sequences appear in Ramsey statements whose aim it is to control the behavior of a coloring on them. We introduce a new 
general notion of such control, which is akin to finding upper bounds on Ramsey degrees, but whose nature is algebraic. (For a definition and applications of 
Ramsey degrees, see, for example, \cite{KPT}.) 
The main result then is Theorem~\ref{T:abst}, which gives control over a coloring on a basic sequence 
from the existence of an appropriate homomorphism. Thus, proving Ramsey statements is reduced to finding homomorphisms. 
Furthermore, as mentioned above, we introduce a natural notion of tensor product for function arrays that allows us to propagate  
the existence of homomorphisms and, therefore, to propagate Ramsey statements.

\subsection{Function arrays and total function arrays}\label{Su:mose}
Here we recall the notion of partial semigroup and, more importantly, we introduce our main Ramsey theoretic structures: 
function arrays, total function arrays and homomorphisms between them. 

As in  \cite{BBH} and \cite{To}, a {\bf partial semigroup} is a set $S$ with a function (operation) from a subset of $S\times S$ to $S$ such that
for all $r,s,t\in S$ if one of the two products $(rs)t$, $r(st)$ is defined, then so is the other and $(rs)t= r(st)$. 
Note that a semigroup is a partial semigroup whose binary operation is total.

Now, let $\Lambda$ be a non-empty set. Let $S$ be a partial semigroup and let $X$ be a set.
By a {\bf function array over $S$ indexed by $\Lambda$ and based on $X$} we understand an assignment
to each $\lambda\in \Lambda$ of a partial function, which we also call $\lambda$, 
from $X$ to $S$ with the property that 
for all $s_0, \dots, s_k\in S$ there exists $x\in X$ such that, for each $\lambda\in \Lambda$, 
\begin{equation}\label{E:dede} 
s_0\lambda(x), \dots, s_k\lambda(x) \hbox{ are all defined.}
\end{equation}
So the domain of each $\lambda\in \Lambda$ is a subset, possibly proper, of $X$; 
condition \eqref{E:dede} means, in particular, that $x$ is in the domain of each $\lambda\in \Lambda$.  
We call a function array as above {\bf total} if $S$ is a semigroup and the domain each 
$\lambda\in \Lambda$ is equal to $X$. 
We call a function array {\bf point based} if $X$ consist of one point, which we usually denote by $\bullet$; 
so $X=\{ \bullet\}$ in this case. Note that, by condition \eqref{E:dede}, the domain of each $\lambda$ is equal to $\{\bullet \}$. 
A point based function array can be, therefore, identified with a function $\Lambda\to S$ given by 
\[
\Lambda\ni \lambda\to \lambda(\bullet)\in S.
\]
Moreover, if $S$ is a semigroup, then a point based function array is automatically total.

We give now some constructions that will be used in Section~\ref{S:conn}. 
Let $S$ be a partial semigroup. As usual, a function $h\colon S\to S$ is an {\bf endomorphism} if for all $s_1,s_2\in S$ with 
$s_1s_2$ defined, $h(s_1)h(s_2)$ is defined and $h(s_1s_2) = h(s_1)h(s_2)$. 
Let $M$ be a monoid.  An action of $M$ on $S$ is called 
is called an {\bf endomorphism action of $M$ on $S$} if, for each $a\in M$, 
the function $s\to a(s)$ is an endomorphism of $S$ and,   
for all $s_1, \dots, s_n\in S$ and each $a\in M$, there is $t\in S$ such that $s_1a(t), \dots, s_na(t)$ are defined. 
Obviously, we will identify such an action with the function $\alpha\colon M\times S\to S$ given by $\alpha(a,s) = a(s)$.

An endomorphism action $\alpha$ of a monoid $M$ on a semigroup $S$ gives rise to two types of function arrays, both of which are over $S$ and indexed by $M$ but are based on different sets. 
The first of these function arrays is based on $S$ and is defined as follows. Let 
\begin{equation}\label{E:sa}
S(\alpha)
\end{equation}
be the function array over $S$ indexed by $M$ and based on $S$ that is obtained by interpreting each $a\in M$ 
as the function from $S$ to $S$ given by the action, that is, 
\[
S\ni s \to \alpha(a,s)\in S.
\]

The second function array arising from $\alpha$ is point based and needs a specification of an element of $S$. So fix ${\bf s}\in S$. Let 
\begin{equation}\label{E:sas}
S(\alpha)_{\bf s}
\end{equation}
be the point based function array over $S$ indexed by $M$ obtained by interpreting each $a\in M$
as the function on $\{ \bullet\}$ given by 
\[
a(\bullet) = \alpha(a, {\bf s}).
\]

Function arrays used in this paper will be of the above form or will be obtained from such by the tensor product operation defined in Section~\ref{Su:ten}.

\subsection{Basic sequences and tame colorings}\label{Su:bata}

Assume we have function arrays $\mathcal S$ and $\mathcal A$ both indexed by $\Lambda$, but with $\mathcal S$ being over a partial semigroup $S$ and based on $X$, while $\mathcal A$ being 
over a semigroup $A$ and point based. To make notation clearer, we use $\vee$ for the binary operation on $A$. We write 
\[
\Lambda(\bullet) =\{ \lambda(\bullet)\colon \lambda\in \Lambda\}\subseteq A.
\]

A sequence $(x_n)$ of elements of $X$ is called {\bf basic in $\mathcal S$} if for all $n_0<\cdots <n_l$ and 
$\lambda_0, \dots, \lambda_l\in \Lambda$ the product
\begin{equation}\label{E:razu}
\lambda_0(x_{n_0}) \lambda_1(x_{n_1})  \cdots \lambda_l(x_{n_l})
\end{equation}
is defined in $S$.

We say that a coloring of $S$ is {\bf $\mathcal A$-tame on} $(x_n)$, where $(x_n)$ is a basic sequence in $\mathcal S$, 
if the color of the elements of the form
\eqref{E:razu} with the additional condition
\begin{equation}\label{E:addi}
\lambda_k(\bullet )\vee \cdots \vee \lambda_l(\bullet )\in \Lambda(\bullet ),\hbox{ for each }k\leq l,
\end{equation}
depends only on the element 
\[
\lambda_0(\bullet )\vee \cdots\vee \lambda_l(\bullet )\in A. 
\]
So, subject to condition \eqref{E:addi}, the color on of the product \eqref{E:razu} 
is entirely controlled by $\lambda_0(\bullet )\vee \cdots\vee \lambda_l(\bullet )$, which, in general, contains much less information than \eqref{E:razu}.

\subsection{Total function arrays from function arrays}

There is a canonical way of associating a total function array to each function array, which generalizes the operation $\gamma S$ of 
compactification of a directed partial semigroup $S$ from \cite{BBH}. 
The definition of $\gamma S$, the semigroup structure and compact topology on it should be recalled here, for example, from \cite[p.31]{To}; in 
particular, we use the symbol $*$ to denote the semigroup operation on $\gamma S$, that is, the product of ultrafilters. 

Let $\mathcal S$ be a function array over a partial semigroup $S$, indexed by $\Lambda$ and based on $X$. Note that \eqref{E:dede} implies that 
$S$ is directed as defined in \cite[p.30]{To}. It follows 
that $\gamma S$ is defined. Let $\gamma X$ be the set of all ultrafilters $\cu$ on $X$ such that for each $s\in S$ and $\lambda\in \Lambda$
\[
\{ x\in X\colon s\lambda(x)\hbox{ is defined}\}\in \cu.
\]
It is clear from condition \eqref{E:dede} that $\gamma X$ is non-empty. It is also easy to verify that $\gamma X$ is compact 
with the usual {\v C}ech--Stone topology on ultrafilters. 
Each $\lambda$ extends to a function, again called $\lambda$, from $\gamma X$ to $\gamma S$ by the usual formula, for
$\cu\in \gamma X$,
\[
A\in \lambda(\cu)\; \hbox{ iff }\;\lambda^{-1}(A)\in \cu.
\]
It is easy to see that the image of each $\lambda$ is, indeed, included in $\gamma S$ and each function $\lambda\colon \gamma X\to \gamma S$ is continous. 
Since $\gamma S$ is a semigroup,
we get a total function array over $\gamma S$, indexed by $\Lambda$ and based on $\gamma X$. 
We denote this total function array by $\gamma{\mathcal S}$. The topologies on 
$\gamma S$ and $\gamma X$ will play a role later on.

\subsection{The Ramsey theorem}

The following natural notion of homomorphism will be crucial in stating Theorem~\ref{T:abst}.  
Assume we have total function arrays $\mathcal A$ and $\mathcal B$ both indexed by $\Lambda$, with $\mathcal A$ being over $A$ and based on $X$ and
$\mathcal B$ being over $B$ and based on $Y$. A {\bf homomorphism from $\mathcal A$ to $\mathcal B$} is a pair of functions 
$(f,g)$ such that
$f\colon X\to Y$, $g\colon A\to B$, $g$ is a homomorphism of semigroups, and, for each $x\in X$ and $\lambda\in \Lambda$, we have
\[
\lambda(f(x)) = g(\lambda(x)).
\]

The following theorem is the main result of Section~\ref{S:str}. 
The bottom line of it is that a homomorphism from a point based function array $\mathcal A$ gives rise to basic 
sequences on which colorings are $\mathcal A$-tame. 

\begin{theorem}\label{T:abst}
Let $\mathcal A$ and $\mathcal S$ be function arrays both indexed by a finite set $\Lambda$, with 
$\mathcal A$ being point based and over a semigroup and $\mathcal S$ being over a partial semigroup $S$. Let
$(f,g)\colon {\mathcal A}\to \gamma {\mathcal S}$ be a homomorphism. Then for each $D\in f(\bullet )$ and
each finite coloring of $S$, there exists a basic sequences $(x_n)$ of elements of $D$ on which the coloring is $\mathcal A$-tame.
\end{theorem}

\begin{proof} Consistently with our conventions, $\vee$ denotes the semigroup operation in the semigroup over which $\mathcal A$ is defined.
Let $\mathcal S$ be based on a set $X$.

Set $\cu = f(\bullet )$. Observe that if $\lambda(\bullet ) = \lambda'(\bullet )$, then $\lambda(\cu) = \lambda'(\cu)$ since
\[
\lambda(f(\bullet )) = g(\lambda(\bullet )) = g(\lambda'(\bullet )) = \lambda'(f(\bullet )).
\]
This observation allows us to define for $\sigma\in \Lambda(\bullet )$,
\[
\sigma (\cu) = \lambda(\cu)
\]
for some, or, equivalently, each, $\lambda\in \Lambda$ with $\lambda(\bullet )=\sigma$. Observe further that for $\sigma\in \Lambda(\bullet )$ we have
\begin{equation}\label{E:sok}
g(\sigma) = \sigma(\cu).
\end{equation}
Indeed, fix $\lambda\in \Lambda$ with $\sigma = \lambda(\bullet )$. Then we have
\[
\sigma(\cu) = \lambda(f(\bullet )) = g(\lambda(\bullet )) = g(\sigma).
\]
For $P\subseteq X$ and $\sigma\in \Lambda(\bullet)$, set
\[
\sigma(P)= \bigcap \{ \lambda(P)\colon \lambda(\bullet ) = \sigma\}.
\]
Note that if $P\in \cu$ and $\lambda\in \Lambda$, then $\lambda(P)\in \lambda(\cu)$ since $P\subseteq \lambda^{-1}(\lambda(P))$.
So, for $\lambda$ with $\lambda(\bullet )=\sigma$, we have $\lambda(P)\in\sigma(\cu)$, and,
therefore, by finiteness of $\Lambda$, we get $\sigma(P)\in \sigma(\cu)$.

Consider a finite coloring of $S$. Let $P\in \cu$ be such that the coloring is constant on $\sigma(P)$ for each $\sigma\in \Lambda(\bullet )$, using
the obvious observation that $\sigma(P)\subseteq \sigma(P')$ if $P\subseteq P'$.

Now, we produce $x_n\in X$ and $P_n\subseteq X$ so that
\begin{enumerate}
\item[(i)] $x_n\in D$, $P_n\subseteq P$;

\item[(ii)] $\lambda_1(x_{m_1})  \lambda_2(u) \in \big(\lambda_1(\bullet )\vee \lambda_2(\bullet )\big)(P_{m_1})$,
for all $m_1<n$, all $u\in P_n$, and all $\lambda_1, \lambda_2\in \Lambda$ with $\lambda_1(\bullet )\vee \lambda_2(\bullet )\in \Lambda(\bullet)$;

\item[(iii)]  the set of $u$ fulfilling the following condition belongs to $\mathcal U$: 
$\lambda_1(x_{m}) \lambda_2(u)\in \big(\lambda_1(\bullet )\vee \lambda_2(\bullet )\big)(P_{m})$, for all $m\leq n$ and all
$\lambda_1, \lambda_2\in \Lambda$ with $\lambda_1(\bullet )\vee \lambda_2(\bullet )\in \Lambda(\bullet)$;

\item[(iv)] $\lambda(x_m)\in \lambda(\bullet )(P_m)$, for all $m\leq n$ and all $\lambda\in \Lambda$.
\end{enumerate}
Note that in points (ii) and (iii) above the condition $\lambda_1(\bullet )\vee \lambda_2(\bullet )\in \Lambda(\bullet)$ ensures that
$\big(\lambda_1(\bullet )\vee \lambda_2(\bullet )\big)(P_{m_1})$ and $\big(\lambda_1(\bullet )\vee \lambda_2(\bullet )\big)(P_{m})$ are defined.

Assume we have $x_m, P_m$ for $m<n$ as above. We produce $x_n$ and $P_n$ so that points (i)--(iv) above hold.
Define $P_n$ by letting
\[
P_n = P \cap \bigcap_{m<n} C_{m},
\]
where $C_m$ consists of those $u\in X$ for which, for all $\lambda_1,\lambda_2\in \Lambda$, 
\[
\begin{split}
\hbox{if }\lambda_1(\bullet)\vee \lambda_2(\bullet)\in \Lambda(\bullet), \hbox{ then } \lambda_1(x_{m})\lambda_2(u) \in
\big(\lambda_1(\bullet )\vee \lambda_2(\bullet )\big)(P_{m}).
\end{split}
\]
For $n=0$, by convention, we set $\bigcap_{m<n} C_{m} = S$. Observe that (ii) holds for $n$.
Our inductive assumption (iii) implies that $C_{m}\in \cu$. Thus,
the definition of $P_n$ gives that $P_0=P\in \cu$ and, for $n>0$, $P_n\in \cu$.

Using \eqref{E:sok} in the last equality, we have that, for all $\lambda_1, \lambda_2\in \Lambda$ with $\lambda_1(\bullet)\vee \lambda_2(\bullet)\in \Lambda(\bullet)$,
\begin{equation}\label{E:lal}
\begin{split}
\lambda_1(\cu)*\lambda_2(\cu) &= \lambda_1(f(\bullet ))*\lambda_2(f(\bullet )) = g(\lambda_1(\bullet )) *g(\lambda_2(\bullet ))\\
&= g(\lambda_1(\bullet )\vee \lambda_2(\bullet ))
= \big( \lambda_1(\bullet )\vee \lambda_2(\bullet )\big)(\cu).
\end{split}
\end{equation}
Separately, we note that $P_n\in \cu$ and therefore, for $\lambda_1, \lambda_2\in \Lambda$ with $\lambda_1(\bullet)\vee \lambda_2(\bullet)\in \Lambda(\bullet)$,
\begin{equation}\label{E:laf}
\begin{split}
\big( \lambda_1(\bullet )\vee \lambda_2(\bullet )\big)(P_n) &\in \big( \lambda_1(\bullet )\vee \lambda_2(\bullet )\big)(\cu)\;\hbox{ and }\\
\lambda_1(\bullet )(P_n) &\in \lambda_1(\bullet )(\cu)= \lambda_1(\cu).
\end{split}
\end{equation}
It follows from \eqref{E:lal} and \eqref{E:laf} that
we can pick $x_n$ for which (iii) and (iv) hold. Since $D\in \cu$, we can also arrange that $x_n\in D$.
So (i) is also taken care of.

Now, it suffices to show that the sequence $(x_n)$ constructed above is as needed. The entries of $(x_n)$ come from
$D$ by (i). By induction on $l$, we show that for all
$m_0<m_1< \cdots  <m_l$ and all $\lambda_0,  \lambda_1, \dots , \lambda_l\in \Lambda$,
we have
\begin{equation}\label{E:G}
\lambda_0(x_{m_0})  \lambda_1(x_{m_1})  \lambda_2(x_{m_2}) \cdots \lambda_l(x_{m_l})
\in  \big(\lambda_0(\bullet ) \vee \lambda_1(\bullet )\cdots \vee \lambda_l(\bullet )\big) (P_{m_0}),
\end{equation}
provided that $\lambda_k(\bullet )\vee \cdots \vee \lambda_l(\bullet )\in \Lambda(\bullet )$ for all $k\leq l$. This claim
will establish the theorem since $P_{m_0}\subseteq P$ by (i).

The case $l=0$ of \eqref{E:G} is (iv). We check the inductive step for \eqref{E:G} using point (ii). Let $l>0$.
Fix $m_0<m_1< \cdots <m_l$ and $\lambda_0,  \lambda_1,  \dots , \lambda_l\in \Lambda$.
By our inductive assumption, we have
\begin{equation}\label{E:new}
\lambda_1(x_{m_1}) \lambda_2(x_{m_2})  \cdots \lambda_l(x_{m_l}) \in \big(\lambda_1(\bullet )\vee \cdots \vee \lambda_l(\bullet )\big)(P_{m_1}).
\end{equation}
Let $\lambda\in \Lambda$ be such that
\begin{equation}\label{E:tot}
\lambda(\bullet ) =\lambda_1(\bullet )\vee \cdots\vee \lambda_l(\bullet ).
\end{equation}
Since
\begin{equation}\notag
\big(\lambda_1(\bullet )\vee \cdots \vee \lambda_l(\bullet )\big)(P_{m_1})\subseteq \lambda(P_{m_1}),
\end{equation}
by \eqref{E:new}, there exists $y\in P_{m_1}$ such that
\begin{equation}\label{E:yyy}
\lambda(y) = \lambda_1(x_{m_1}) \lambda_2(x_{m_2})  \cdots \lambda_l(x_{m_l}).
\end{equation}
Since $m_0<m_1$ and since $y\in P_{m_1}$, from (ii) with $n=m_1$, we get
\begin{equation}\label{E:hewn}
\lambda_0(x_{m_0}) \lambda(y) \in \big(\lambda_0(\bullet )\vee \lambda(\bullet )\big)(P_{m_0}).
\end{equation}
Note that (ii) can be applied here as $\lambda_0(\bullet)\vee \lambda(\bullet)\in \Lambda(\bullet)$ as
\[
\lambda_0(\bullet)\vee \lambda(\bullet) = \lambda_0(\bullet)\vee \lambda_1(\bullet )\vee \cdots\vee \lambda_l(\bullet ).
\]
Now \eqref{E:G} follows from \eqref{E:hewn} together with \eqref{E:tot} and \eqref{E:yyy}.
\end{proof}

In the proof above, at stage $n$, $x_n$ is chosen arbitrarily from sets belonging to $f(\bullet )$. It follows that
if $f(\bullet )$ is assumed to be non-principal, then the sequence $(x_n)$ can be chosen to be injective.

\subsection{Tensor product of function arrays}\label{Su:ten}

We introduce and apply a natural notion of tensor product for function arrays. 

Let $\Lambda_0$, $\Lambda_1$ be finite sets. Let
\[
\Lambda_0\star \Lambda_1= \Lambda_0\cup \Lambda_1\cup (\Lambda_0\times \Lambda_1),
\]
where the union is taken to be disjoint. Fix a partial semigroup $S$.
Let ${\mathcal S}_i$, $i=0, 1$, be function arrays over $S$ indexed by $\Lambda_i$ and based on $X_i$, respectively.
Define
\[
{\mathcal S}_0\otimes {\mathcal S}_1
\]
to be the function array over $S$ indexed by $\Lambda_0\star \Lambda_1$ and based on $X_0\times X_1$ defined as follows. With $\lambda_0\in \Lambda_0$,
$\lambda_1\in \Lambda_1$, and $(\lambda_0,\lambda_1)\in \Lambda_0\times \Lambda_1$, we associate partial functions from $X_0\times X_1$ to $S$ by letting
\[
\lambda_0(x_0,x_1) = \lambda_0(x_0),\;\; \lambda_1(x_0,x_1) = \lambda_{1}(x_{1}),\;\; (\lambda_0,\lambda_1)(x_0,x_1) = \lambda_0(x_0) \lambda_{1}(x_{1}),
\]
where the product on the right hand side of the last equality is computed in $S$ and the left hand side is declared to be defined whenever this product exists.
To check that the object ${\mathcal S}_0\otimes {\mathcal S}_1$ defined above is indeed a function array one needs to see condition \eqref{E:dede}. To see this fix 
$s_0, \dots , s_m\in S$. Pick $x_0\in X_0$ with $s_i\lambda(x_0)$ defined for all $i\leq m$ and $\lambda\in \Lambda_0$. Now pick $x_1\in X_1$ with $s_i\lambda'(x_1)$ 
and $(s_i\lambda(x_0)) \lambda'(x_1)$ are defined for all $i\leq m$, $\lambda\in \Lambda_0$ and $\lambda'\in \Lambda_1$. 
Then clearly $s_i\vec{\lambda}(x_0,x_1)$ are defined for all $i\leq m$ and $\vec{\lambda} \in \Lambda_0\star\Lambda_1$, as required.

It is clear that the operation of tensor product is associative; if each ${\mathcal S}_i$, $i<n$, is a function array over $S$ indexed by $\Lambda_i$, then
$\bigotimes_{i<n}{\mathcal S}_i$ is a function array over $S$ indexed by $\Lambda_0\star\cdots \star \Lambda_{n-1}$, 
which consists of all sequences $(\lambda_0, \dots , \lambda_m)$, where 
$\lambda_i\in \Lambda_{j_i}$ for some $j_0<\cdots <j_m<n$.
Note that if each ${\mathcal S}_i$ is point based, then so is the tensor product.

\begin{proposition}\label{P:pro1}
Let $S$ be a partial semigroup.
Let ${\mathcal S}_i$, $i=0, 1$, be function arrays over $S$ based on $\Lambda_i$, respectively.
Then there is a homomorphism 
\[
\gamma {\mathcal S}_0\otimes\gamma {\mathcal S}_1 \to \gamma({\mathcal S}_0 \otimes {\mathcal S}_1). 
\]
\end{proposition}

\begin{proof} Let ${\mathcal S}_0$ be based on $X_0$ and ${\mathcal S}_1$ on $X_1$. 
Then $\gamma {\mathcal S}_0\otimes\gamma {\mathcal S}_1$ is based on $\gamma X_0\times \gamma X_1$, while 
$\gamma({\mathcal S}_0 \otimes {\mathcal S}_1)$ on $\gamma(X_0\times X_1)$. 
Consider the natural map $\gamma X_0\times \gamma X_1\to \gamma(X_0\times X_1)$ given by 
\[
(\cu, \cv)\to \cu\times \cv, 
\] 
where, for $C\subseteq X_0\times X_1$,  
\[
C\in \cu\times \cv \Longleftrightarrow \{ x_0\in X_0\colon \{ x_1\in X_1\colon (x_0,x_1)\in C\}\in \cv \}\in \cu.
\]
Then 
\[
(f, {\rm id}_{\gamma S})\colon \gamma {\mathcal S}_0\otimes\gamma {\mathcal S}_1 \to \gamma({\mathcal S}_0 \otimes {\mathcal S}_1), 
\]
where $f(\cu, \cv)= \cu\times \cv$, is the desired homomorphism. 
\end{proof}

\begin{proposition}\label{P:pro2}
Fix semigroups $A$ and $B$.
For $i=0,1$, let ${\mathcal A}_i$ and ${\mathcal B}_i$ be total function arrays over $A$ and $B$, respectively, indexed by $\Lambda_i$.
Let $(f_i, g)\colon {\mathcal A}_i\to {\mathcal B}_i$ be homomorphisms. Then
\[
(f_0\times f_1, g)\colon {\mathcal A}_0 \otimes {\mathcal A}_1 \to {\mathcal B}_0 \otimes {\mathcal B}_1
\]
is a homomorphism.
\end{proposition}

\begin{proof} Let ${\mathcal A}_i$ be based on a set $X_i$ for $i=0,1$. For $(x_0,x_1)\in X_0\times X_1$ and 
$\vec{\lambda}\in \Lambda_0*\Lambda_1 = \Lambda_0\cup\Lambda_1\cup(\Lambda_0\times \Lambda_1)$, we need to check 
\[
g(\vec{\lambda}(x_0,x_1)) = \vec{\lambda}\big((f_0\times f_1)(x_0,x_1)\big).
\]
We do it only for $\vec{\lambda}\in\Lambda_0\times \Lambda_1$, the case $\vec{\lambda}\in\Lambda_0\cup \Lambda_1$ being essentially identical. 
So for 
$\vec{\lambda}\in \Lambda_0\times \Lambda_1$, we have
\[
\begin{split}
g(\vec{\lambda}(x_0,x_1)) &= g(\lambda_0(x_0)\lambda_{1}(x_{1})) = g(\lambda_0(x_0))g(\lambda_{1}(x_{1}))\\
&=\lambda_0(f_0(x_0)) \lambda_{1}(f_{1}(x_{1})) = \vec{\lambda}\big((f_0\times f_1)(x_0,x_1)\big),
\end{split}
\]
where the second equality holds since $g$ is a homomorphism of semigroups and the third equality holds since each $(f_i,g)$ is a homomorphism
from ${\mathcal A}_i$ to ${\mathcal B}_i$. 
\end{proof}

\subsection{Propagation of homomorphisms}

The first application of tensor product has to do with relaxing condition \eqref{E:addi}. This is done in condition \eqref{E:addir}. 

Let $\mathcal A$ be a point based function array 
over a semigroup $A$ and indexed by $\Lambda$. As before, we denote by $\vee$ the binary operation on $A$.
Let $F$ be a subset of $A$, 
let $\mathcal S$ be a function array over a partial semigroup $S$ indexed by $\Lambda$, and let $(x_n)$ be a basic sequence in 
$\mathcal S$. A coloring of $S$ is said to be {\bf $F$-$\mathcal A$-tame on $(x_n)$} if the color of elements of the form
\eqref{E:razu} with the additional condition
\begin{equation}\label{E:addir}
\lambda_k(\bullet )\vee \cdots\vee \lambda_l(\bullet )\in F,\hbox{ for each }k\leq l,
\end{equation}
depends only on the element $\lambda_0(\bullet )\vee \cdots\vee \lambda_l(\bullet )$ of $A$. The notion of $F$-${\mathcal A}$-tameness becomes 
${\mathcal A}$-tameness if $F= \Lambda(\bullet)$. In applications, it will be desirable to take $F$ strictly larger than $\Lambda(\bullet)$ thus making 
$F$-${\mathcal A}$-tameness strictly stronger. 

The following corollary is a strengthening of Theorem~\ref{T:abst}, but it follows from that theorem via 
the tensor product construction.

\begin{corollary}\label{C:abstsup}
Let $\mathcal A$ and $\mathcal S$ be function arrays both indexed by a finite set $\Lambda$, with 
$\mathcal A$ being point based and over a semigroup $A$ and $\mathcal S$ being over a partial semigroup $S$. Let
$(f,g)\colon {\mathcal A}\to \gamma {\mathcal S}$ be a homomorphism. Then for each $D\in f(\bullet )$, 
each finite coloring of $S$, and each finite set $F\subseteq A$, 
there exists a basic sequences $(x_n)$ of elements of $D$ on which the coloring is $F$-$\mathcal A$-tame.
\end{corollary}

\begin{proof} For a natural number $r>0$, by $\Lambda_{<r}$ 
we denote the set of all sequences $\vec{\lambda} = (\lambda_0, \dots, \lambda_m)$ of elements of $\Lambda$ with $m<r$. 
We associate with each such $\vec{\lambda}$ an element $\vec{\lambda}(\bullet )$ of $A$ by letting
\[
\vec{\lambda}(\bullet ) = \lambda_0(\bullet )\vee \cdots\vee \lambda_{m}(\bullet ). 
\]

Since $F$ is finite, there exists $r$ such that
\[
F\cap \{ \vec{\lambda}(\bullet)\colon \vec{\lambda}\in \bigcup_{r'\geq 1} \Lambda_{<r'}\} \subseteq \{ \vec{\lambda}(\bullet)\colon \vec{\lambda}\in \Lambda_{<r}\}.  
\]
Thus, fixing this $r$, it suffices to show the corollary for $F= \{ \vec{\lambda}(\bullet)\colon \vec{\lambda}\in \Lambda_{<r}\}$.  
We consider the function arrays ${\mathcal A}^{\otimes r}$ and $(\gamma {\mathcal S})^{\otimes r}$ indexed by $\Lambda^{*r}$. 
Note that ${\mathcal A}^{\otimes r}$ is based on a set consisting of only the $r$-tuple $(\bullet, \dots, \bullet)$ and that 
\[
\Lambda^{*r}(\bullet, \dots, \bullet) = \{ \vec{\lambda}(\bullet)\colon \vec{\lambda}\in \Lambda_{<r}\}.
\]
Further, by Proposition~\ref{P:pro2}, 
there exists a homomorphism from ${\mathcal A}^{\otimes r}$ to $(\gamma {\mathcal S})^{\otimes r}$, which
is equal to $(f^r, g)$, where $f^r$ stands for the $r$-fold product $f\times \cdots \times f$. 
Note also that $D\times X^{r-1}\in f^r(\bullet )$. Since, by Proposition~\ref{P:pro1}, there is a homomorphism from 
$(\gamma {\mathcal S})^{\otimes r}$ to $\gamma ({\mathcal S}^{\otimes r})$, we are done by Theorem~\ref{T:abst} by composing the two homomorphisms.
\end{proof}

We have one more corollary of Theorem~\ref{T:abst} and the tensor product construction. It concerns double sequences. 
Let ${\mathcal S}$ be a function array over $S$ indexed by $\Lambda$ and based on $X$.
A double sequence $(x_n,y_n)$ of elements of $X$ will be called {\bf basic} if the single sequence
\[
x_0, y_0, x_1, y_1, x_2, y_2, \dots
\]
is basic. Having a basic sequence $(x_n, y_n)$, we will be interested in controlling the color on elements of the form
\begin{equation}\label{E:dwa}
\begin{split}
\lambda_0(x_{m_0})\lambda'_0(y_{n_0})  \lambda_1(x_{m_1})& \lambda'_1(y_{n_1}) \lambda_2(x_{m_2}) \lambda'_2(y_{n_2}) \cdots
\lambda_l(x_{m_l})\lambda'_l(y_{n_l})\\
&\hbox{and}\\
\lambda_0(x_{m_0})\lambda'_0(y_{n_0}) \lambda_1(x_{m_1})& \lambda'_1(y_{n_1}) \lambda_2(x_{m_2}) \lambda'_2(y_{n_2}) \cdots
\lambda_l(x_{m_l})
\end{split}
\end{equation}
for $m_0\leq n_0<m_1\leq n_1< \cdots <m_l\leq n_l$ and
$\lambda_0, \lambda'_0, \dots , \lambda_l, \lambda'_l\in \Lambda$.

Let $\mathcal A$ be a point based total function array over a semigroup $A$ indexed by $\Lambda$.
For $\lambda, \lambda'\in \Lambda$, we say that $\lambda$ and $\lambda'$ are {\bf conjugate} if $\lambda(\bullet ) = \lambda'(\bullet )$.
For a set $F\subseteq A$, we say that a coloring of $S$ is {\bf conjugate $F$-${\mathcal A}$-tame on} $(x_n, y_n)$
if the color of the elements of the form \eqref{E:dwa} such that for each $k\leq l$
\[
\lambda_k\hbox{ and }\lambda_k'\hbox{ are conjugate and }\;\lambda_k(\bullet )\vee \cdots \vee \lambda_l(\bullet )\in F
\]
depends only on
\[
\lambda_0(\bullet )\vee \cdots \vee \lambda_l(\bullet )\in A.
\]

\begin{corollary}\label{C:abst}
Let $\mathcal A$ and $\mathcal S$ be function arrays both indexed by a finite set $\Lambda$, with 
$\mathcal A$ being point based and over a semigroup $A$ and $\mathcal S$ being based on $X$ and over a partial semigroup $S$. 
Let $(f,g)\colon {\mathcal A}\to \gamma {\mathcal S}$ be a homomorphism. 
Let $v\in \gamma X$ be such that for each $\lambda\in \Lambda$
\begin{equation}\label{E:eqa}
\lambda(f(\bullet )) \lambda(v) = \lambda(f(\bullet )).
\end{equation}
Then, for each finite subset $F$ of $A$, $D\in f(\bullet )$ and $E\in v$ and each finite coloring of $S$, there exists a basic sequence $(x_n,y_n)$, with $x_n\in D$ and 
$y_n\in E$, on which the coloring is conjugate $F$-$\mathcal A$-tame.
\end{corollary}

\begin{proof} 
Let $\mathcal A$ be based on $\{ \bullet\}$. Let
\[
\Lambda' = \Lambda\cup \{ (\lambda_0, \lambda_1)\in \Lambda\times \Lambda\colon \lambda_0(\bullet ) = \lambda_1(\bullet )\}.
\]
Fix a point $\bullet'$.
Let ${\mathcal A}'$ be the point based function array over $A$ and indexed by $\Lambda'$ that is based on the point $(\bullet , \bullet')$ and such that
for $\lambda,\, (\lambda_0, \lambda_1)\in \Lambda'$
\[
\lambda(\bullet ,\bullet') = \lambda(\bullet )\;\hbox{ and }\;(\lambda_0, \lambda_1)(\bullet ,\bullet' ) = \lambda_0(\bullet ).
\]

The index set of $\gamma({\mathcal S}\otimes {\mathcal S})$ and $\gamma {\mathcal S}\otimes \gamma{\mathcal S}$ is $\Lambda*\Lambda 
= \Lambda_0\cup\Lambda_1\cup (\Lambda_0\times \Lambda_1)$, where $\Lambda_0=\Lambda_1=\Lambda$ and 
the union is disjoint. We view $\Lambda'$ as included in 
$\Lambda*\Lambda$ with the copy of $\Lambda$ in $\Lambda'$ identified with $\Lambda_0$ in $\Lambda*\Lambda$. This inclusion 
allows us to consider $\gamma({\mathcal S}\otimes {\mathcal S})$ and $\gamma {\mathcal S}\otimes \gamma{\mathcal S}$ with the index set restricted to $\Lambda'$, 
which we do below. 
Clearly, any homomorphism $\gamma {\mathcal S}\otimes \gamma{\mathcal S}\to \gamma ({\mathcal S}\otimes {\mathcal S})$, 
when the two function arrays are indexed by $\Lambda*\Lambda$, is also a homomorphisms when they are indexed by $\Lambda'$.

Now, to obtain the conclusion of the corollary, 
by Corollary~\ref{C:abstsup}, it suffices to produce a homomorphism ${\mathcal A}'\to \gamma ({\mathcal S}\otimes {\mathcal S})$ such that 
$D\times E$ is in the image of $(\bullet, \bullet')$. 
Let $f'\colon \{ \bullet'\}\to \gamma X$ be given by $f'(\bullet') = v$.
Since, by Proposition~\ref{P:pro1}, there is a homomorphism
\[
(\rho, \pi)\colon \gamma {\mathcal S}\otimes \gamma{\mathcal S}\to \gamma ({\mathcal S}\otimes {\mathcal S})
\]
and $(D,E)\in \rho\circ (f\times f')(\bullet, \bullet')$, it suffices to show that
\[
(f\times f', g)\colon {\mathcal A}'\to \gamma {\mathcal S}\otimes \gamma{\mathcal S}
\]
is a homomorphism. This amounts to showing that if $\lambda, \lambda_0,\lambda_1\in \Lambda$ and $\lambda_0(\bullet ) = \lambda_1(\bullet )$, 
then we have the following two equalities 
\begin{equation}\label{E:need}
\begin{split} 
\lambda((f\times f')(\bullet , \bullet')) &=  g(\lambda(\bullet ,\bullet' )),\\
(\lambda_0, \lambda_1)((f\times f')(\bullet , \bullet' )) &=  g((\lambda_0, \lambda_1)(\bullet ,\bullet' )).
\end{split}
\end{equation}
We verify only the second equality, the first one being easier.
To start, we note that
\[
\lambda_0(f(\bullet )) = g(\lambda_0(\bullet )) = g(\lambda_1(\bullet )) = \lambda_1(f(\bullet )).
\]
Using this equality and \eqref{E:eqa}, we check the second equality in \eqref{E:need} by a direct computation as follows
\[
\begin{split}
(\lambda_0, \lambda_1)((f\times f')(\bullet , \bullet' )) &= (\lambda_0, \lambda_1)(f(\bullet ), f'(\bullet')) =
(\lambda_0, \lambda_1)(f(\bullet ), v)\\
&= \lambda_0(f(\bullet ))\lambda_1(v) = \lambda_1(f(\bullet )) \lambda_1(v)\\
&= \lambda_1(f(\bullet )) = \lambda_0(f(\bullet )) = g(\lambda_0(\bullet ))\\
&= g((\lambda_0, \lambda_1)(\bullet,\bullet' )).\qedhere
\end{split}
\]
\end{proof}

\section{Monoid actions and infinitary Ramsey theorems}\label{S:conn}

{\em In this section, $M$ will be a finite monoid.} 

We connect here the dynamical result of Section~\ref{S:moa} with the algebraic/Ramsey theoretic result of Section~\ref{S:str}. This connection is made possible 
by Corollary~\ref{C:yet}, which translates Theorem~\ref{T:main} (via Corollary~\ref{C:rest}) into a statement about 
the existence of a homomorphism needed for applications of Theorem~\ref{T:abst}. Ramsey theoretic consequences of Corollary~\ref{C:yet} 
are investigated later in the section, in particular, a general Ramsey theoretic result, Corollary~\ref{C:genseq}, is derived from it. 
We also introduce the notion of Ramsey monoid and give a 
characterization of those among almost R-trivial monoids. We use this characterization to determine 
which among the monoids $I_n$ from Section~\ref{Su:exa} are Ramsey. This result implies an answer to Lupini's question \cite{Lu}
on possible extensions of Gowers' theorem. We derive some concrete Ramsey results from our general 
considerations. For example, we obtain the Furstenberg--Katznelson Ramsey theorem for located words.

\subsection{Connecting Theorems~\ref{T:main} and \ref{T:abst}}

In order to apply Theorem~\ref{T:abst}, through Corollaries~\ref{C:abstsup} and \ref{C:abst},  
one needs to produce appropriate homomorphisms between function arrays. 
We show how Theorem~\ref{T:main}, through Corollary~\ref{C:rest},  
gives rise to exactly such homomorphisms. This is done in Corollary~\ref{C:yet} below. 
All the Ramsey theorems are results of combining Corollaries~\ref{C:abstsup} and \ref{C:abst} 
with Corollary~\ref{C:yet}.

Let $S$ be a partial semigroup. For $A\subseteq S$, we say that $S$ is {\bf $A$-directed} if for all $x_1, \dots, x_n\in S$ there exists
$x\in A$ such that $x_1x, \dots, x_n x$ are all defined. So $S$ is directed as defined in \cite{To} if it is $S$-directed. We say that $I\subseteq S$
is a {\bf two-sided ideal in} $S$ if it is non-empty and, for $x, y\in S$ for which $xy$ is defined, $xy\in I$ if $x\in I$ or $y\in I$.

Recall the definitions of the function arrays from \eqref{E:sa} and \eqref{E:sas} in Section~\ref{Su:mose}. 
Recall also from \eqref{E:orsem} the endomorphism action of $M$ on the semigroup $\langle{\mathbb Y}(M)\rangle$ generated 
by ${\mathbb Y}(M)$ as in Section~\ref{Su:sempar}. Denoting this natural action by $\beta$ and taking $y_0\in {\mathbb Y}(M)$, we form 
the point based total function array $\langle {\mathbb Y}(M)\rangle(\beta)_{y_0}$. For the sake of simplicity, we 
denote it by 
\[
\langle{\mathbb Y}(M)\rangle_{y_0}.
\]

The following corollary will be seen to be a consequence of Corollary~\ref{C:rest}. 

\begin{corollary}\label{C:yet}
Assume $M$ is almost R-trivial. Let ${\mathbf y}\in {\mathbb Y}(M)$ be a maximal element,  
and let $\alpha$ be an endomorphism action of $M$ on a partial semigroup $S$. 
Let $I\subseteq S$ be a two-sided ideal such that $S$ is $I$-directed. 
\begin{enumerate}
\item[(i)] There exists a homomorphism $(f,g)\colon \langle {\mathbb Y}(M)\rangle_{\mathbf y} \to \gamma (S(\alpha))$ with 
\[
I\in f(\bullet).  
\]

\item[(ii)] Assume that $M$ is the monoid $J(\emptyset, B)$ for a finite set $B$ as defined in Section~\ref{Su:exa}. 
If $E\subseteq S$ is a right ideal such that $S$ is $E$-directed, 
then there exists a homomorphism $(f,g)\colon \langle {\mathbb Y}(M)\rangle_{\mathbf y} \to \gamma (S(\alpha))$ and an ultrafilter ${\mathcal V} \in \gamma(S)$ such that 
\[
I\in f(\bullet)
\]  
and, additionally, 
\[
E\in {\mathcal V}\;\hbox{ and }\; f(\bullet) * {\mathcal V} = f(\bullet).
\]
\end{enumerate}
\end{corollary}

To prove Corollary~\ref{C:yet}, we will need the following lemma, whose proof is standard. 

\begin{lemma}\label{L:dire}
Let $S$ be a partial semigroup. 
\begin{enumerate}
\item[(i)] Let $I$ be a two-sided 
ideal in $S$ such that $S$ is $I$-directed.
Then $\{ \cu\in \gamma S\colon I\in \cu\}$ is a compact two-sided 
ideal in $\gamma S$.

\item[(ii)] Let $E$ be a right 
ideal in $S$ such that $S$ is $I$-directed.
Then $\{ \cu\in \gamma S\colon E\in \cu\}$ is a compact right 
ideal in $\gamma S$.

\end{enumerate}
\end{lemma}

\begin{proof} We give an argument only for (i); the argument for (ii) being similar. 

For $x\in S$, let $S/x=\{ y\in S\colon xy \hbox{ is defined}\}$. 

Let $H= \{ \cu\in \gamma S\colon I\in \cu\}$. Then, by definition, $H$ is clopen. It is non-empty
since, by $I$-directedness of $S$, the family $\{ I\}\cup \{ S/x\colon x\in S\}$ of subsets of $S$
has the finite intersection property, so it is contained in an ultrafilter, which is necessarily an element of $H$.

We check that $I\in \cu * {\mathcal V}$ if  $I\in \cu$ or $I\in {\mathcal V}$.
Assume first that $I\in \cu$. For $x\in I$, $S/x\subseteq \{ y\colon xy\in I\}$, therefore,
since $S/x\in {\mathcal V}$, for each $x\in I$, we have $\{ y\colon xy\in I\} \in {\mathcal V}$.
So $I\subseteq \{ x\colon \{ y\colon xy\in I\} \in {\mathcal V}\}$. Since $I\in \cu$, we get
\[
\{ x\colon \{ y\colon xy\in I\} \in {\mathcal V}\} \in \cu
\]
which means $I\in \cu* {\mathcal V}$. Assume now $I\in {\mathcal V}$. For each $x\in S$, we have
$I\cap (S/x)\subseteq \{ y\colon xy\in I\}$. Therefore, since $I, S/x\in {\mathcal V}$, we have
$\{ y\colon xy\in I\} \in {\mathcal V}$ for each $x\in S$.
So 
\[
\{ x\colon \{ y\colon xy\in I\} \in {\mathcal V}\} = S\in \cu, 
\]
which means $I\in \cu* {\mathcal V}$.
\end{proof}

\begin{proof}[Proof of Corollary~\ref{C:yet}]
(i) We denote by $\gamma I$ the compact two sided ideal $\{ \cu\in \gamma S\colon I\in \cu\}$ from Lemma~\ref{L:dire}. 

Observe that the action $\alpha$ naturally induces an action of $M$ by continuous endomorphisms on $\gamma S$. 
We call this resulting action $\gamma\alpha$. By Corollary~\ref{C:rest}, there exists a homomorphism 
$g\colon \langle {\mathbb Y}(M)\rangle\to \gamma S$ such that all maximal elements of ${\mathbb Y}(M)$ are mapped 
to $I(S)$. In particular, $g({\mathbf y})\in I(\gamma S)$. Since, by Lemma~\ref{L:dire}, $\gamma I$ is a compact two-sided ideal, 
we have $I(\gamma S)\subseteq \gamma I$. Thus, $g({\mathbf y})\in \gamma I$, that is, 
\[
I\in g({\mathbf y}). 
\]
Note now that if we let $f(\bullet) = g({\mathbf y})$, then $(f,g)\colon \langle {\mathbb Y}(M)\rangle_{\mathbf y}\to (\gamma S)(\gamma\alpha)$ 
is a homomorphism. 
A quick check of definitions gives $(\gamma S)(\gamma\alpha) = \gamma (S(\alpha))$. Thus, $(f,g)$ is as required.

(ii) We proceed constructing $g$ and $f$ as in point (i) above. So we have $I\in f(\bullet)$. 
We note that ${\mathbb X}(M)$ consists of two points: the R-class of $1_M$ and the R-class consisting of all elements of $B$. Thus, 
by the last sentence of Corollary~\ref{C:rest}, $f(\bullet)$ is a minimal idempotent in $\gamma(S)$. Now consider 
$J= \{ \cu\in \gamma S\colon E\in \cu\}$, which by Lemma~\ref{L:dire} is a compact right ideal in $\gamma(S)$. By 
Proposition~\ref{P:fa}(iii), there exists an idempotent ${\mathcal V} \in J$ with $f(\bullet) * {\mathcal V} = f(\bullet)$. Since ${\mathcal V}\in J$, 
we have $E\in {\mathcal V}$, which completes the proof. 
\end{proof}

\subsection{Ramsey theorems from monoids}\label{Su:ramo}

Given a sequence $(X_n)$ of sets, let
\[
\langle (X_n)\rangle
\]
consist of all finite sequences $x_{1}\cdots x_{k}$ for $k\in {\mathbb N}$, for which there exist $m_1<\cdots < m_k$ such that $x_i\in X_{m_i}$. 
We make this set into a partial semigroup by declaring the product $(x_{1}\cdots x_{k}) (y_{1}\cdots y_{l})$ of two such sequences defined 
if there exist $m_1<\cdots <m_k<n_1<\cdots <n_l$ such that $x_i\in X_{m_i}$ and $y_i\in X_{n_i}$ and then letting the product be equal to 
the concatenation $x_{1}\cdots x_{k}y_{1}\cdots y_{l}$ of the two sequences. Two special cases of this construction are usually considered. 

(1) We fix a set $X$, and let $X_n=X$ for each $n$. In this case, $\langle (X_n)\rangle$ is a semigroup, not just a partial semigroup, 
of all {\bf words in $X$} with concatenation as the semigroup     
operation. 

(2) We fix a set $X$, and let $X_n = \{ n\}\times X$ for each $n$. In this case, $\langle (X_n)\rangle$ is isomorphic to the partial semigroup 
of {\bf located words in $X$}, that is, all 
partial functions from $\mathbb N$ to $X$ with finite domains, where two such functions $f$ and $g$ have their product defined if all elements of the domain of $f$ are 
smaller than all elements of the domain of $g$, and the product is set to be the function whose graph is the union of the graphs of $f$ and $g$. The identification 
\[
\langle (X_n)\rangle\ni (m_1, x_1)\cdots (m_k, x_k)\to f
\]
where the domain of $f$ is $\{ m_1, \dots , m_k\}$ and $f(m_i) = x_i$, establishes a canonical isomorphism between $\langle (X_n)\rangle$ and 
the partial semigroup of located words. 

Since making the assumptions as in (1) or (2) causes no simplification in our arguments, we work with the general definition of $\langle (X_n)\rangle$ as above. 
We only note that a Ramsey statement (like the ones given 
later in this paper) formulated for located words as in (2) is stronger than the analogous statement formulated for words as in (1). 
One derives the latter from the former by applying the map associating with 
a located word $f$ defined on a set $\{ m_1, \dots , m_k\}$ the word $f(m_1)\cdots f(m_k)$.

A {\bf pointed $M$-set} is a set $X$ equipped with an action of $M$ and a distinguished point 
$x$ such that $Mx=X$. Let $(X_n)$ be a sequence of pointed $M$-sets. The monoid $M$ acts on $\langle (X_n)\rangle$ in the natural manner: 
\[
a(x_1 \cdots x_k) = a(x_1)\cdots a(x_k),\hbox{ for }a\in M\hbox{ and }x_1\cdots x_k\in \langle (X_n)\rangle. 
\]
Note that since $x_i\in X_{m_i}$ implies that $a(x_i)\in X_{m_i}$, the action above 
of each element of $M$ is defined on each element of $\langle (X_n)\rangle$. It is clear that this is an endomorphism action. 

A sequence $(w_i)$ of elements of $\langle (X_n)\rangle$ is called {\bf basic in} $\langle (X_n)\rangle$ 
if for all $i_1<\cdots <i_k$, the product $w_{i_1}\cdots w_{i_k}$ is defined. We note that 
if each $X_n$ is a pointed $M$-set for a monoid $M$ and the sequence $(w_i)$ is basic, then for $i_1<\cdots <i_k$ and $a_1, \dots , a_k\in M$ the product 
$a_1(w_{i_1})\cdots a_k(w_{i_k})$ is defined. For this reason no confusion will arise from using the word basic to describe certain sequences in $\langle (X_n)\rangle$ 
and certain sequences in base sets of function arrays over $\langle (X_n)\rangle$ as in Section~\ref{Su:bata}. 

We will register the following general result that follows from 
Corollary~\ref{C:abstsup} and Corollary~\ref{C:yet}(i). This result allows us to control the color of words as in \eqref{E:aaww} by their ``type" as in \eqref{E:ttww}.  

\begin{corollary}\label{C:genseq}
Assume $M$ is almost R-trivial. Let $F$ be a finite subset of the semigroup $\langle {\mathbb Y}(M)\rangle$, whose operation is denoted by $\vee$.  Let ${\mathbf y}\in {\mathbb Y}(M)$
be a maximal element of the forest ${\mathbb Y}(M)$. Given pointed $M$-sets  
$X_n$, $n\in {\mathbb N}$, for each finite coloring of $\langle (X_n)\rangle$, 
there exist a basic sequence $(w_i)$ in $\langle (X_n)\rangle$ such that 
\begin{enumerate}
\item[(i)] for each $i$, $w_i$ contains the distinguished point of some $X_n$ and 

\item[(ii)] for each $i_0<\cdots < i_k$ and $a_0, \dots , a_k\in M$, the color of 
\begin{equation}\label{E:aaww}
a_0(w_{i_0})\cdots a_k(w_{i_k})
\end{equation}
depends only on 
\begin{equation}\label{E:ttww}
a_0({\mathbf y})\vee \cdots \vee a_k({\mathbf y}) 
\end{equation}
provided $a_0({\mathbf y})\vee \cdots \vee a_k({\mathbf y})\in F$. 
\end{enumerate}
\end{corollary}

\begin{proof}
We regard $S= \langle (X_n)\rangle$ as a partial semigroup with concatenation as a partial semigroup operation and 
with the natural action $\alpha$ of $M$. This leads to the function array $S(\alpha)$. Let $I$ be the subset of $\langle (X_n)\rangle$ 
consisting of all words that contain a distinguished element of some $X_n$. It is clear that $I$ is a two-sided ideal 
and that $\langle (X_n)\rangle$ is $I$-directed. 
By Corollary~\ref{C:yet}(i), there exists a homomorphism $(f,g)\colon \langle {\mathbb Y}(M)\rangle_{\mathbf y} \to \gamma S(\alpha)$ 
with $I\in f(\bullet)$.

It is evident from the definition of $\langle {\mathbb Y}(M)\rangle$ that we can find a finite set $F'\subseteq \langle {\mathbb Y}(M)\rangle$ such that if $z_0, \dots, z_l\in  {\mathbb Y}(M)$ and 
$z_0\vee \cdots \vee z_l\in F$, then $z_k\vee\cdots \vee z_l\in F'$ for each $0\leq k\leq l$.
Now, from the existence of the homomorphism $(f,g)$, by Corollary~\ref{C:abstsup}, we get the existence of a basic sequence 
$(w_i)$ in $I$ such that, for $i_0<\cdots<i_l$ and $a_0, \dots, a_l\in M$,  
the color of $a_0(w_{i_0})\cdots a_l(w_{i_l})$ depends only on 
$a_0(\bullet)\vee \cdots \vee a_l(\bullet)$ as long as $a_k(\bullet )\vee \cdots \vee a_l(\bullet )\in F'$, for each $0\leq k\leq l$. 
Since this last condition is implied by $a_0(\bullet )\vee \cdots \vee a_l(\bullet )\in F$ and  
since for each $a\in M$, $a(\bullet) = a({\mathbf y})$, we are done. 
\end{proof}

We also have the following result analogous to Corollary~\ref{C:genseq} for the monoid $J(\emptyset ,B)$, as defined in Section~\ref{Su:exa}, 
that follows from Corollary~\ref{C:abstsup} and Corollary~\ref{C:yet}(ii). 

\begin{corollary}\label{C:genj}
Consider the monoid $J(\emptyset, B)$ for a finite set $B$. 
Let $X_n$, $n\in {\mathbb N}$ be pointed $J(\emptyset, B)$-sets, and let $E$ be a right ideal such that $\langle (X_n)\rangle$ is $E$-directed. 
For each finite coloring of $\langle (X_n)\rangle$, 
there exist a basic double sequence $(v_i, v_i')$ in $\langle (X_n)\rangle$, with $x$ occurring in each $v_i$ and with $v_i'\in E$ for each $i$, such that 
the coloring is fixed on sequences of the form 
\begin{equation}\label{E:refix}
b_0(v_{i_0}) b'_0(v'_{j_0}) b_1(v_{i_1}) b'_1(v'_{j_1}) \cdots b_n(v_{i_n}) 
\end{equation}
for $b_0, b_0', b_1, b_1', \dots, b_n\in B$ and $i_0 \leq j_0 < i_1 \leq j_1< \cdots <i_n$. 
\end{corollary}

\begin{proof}
We start with analyzing ${\mathbb X}(M)$, ${\mathbb Y}(M)$, and $\langle {\mathbb Y}(M)\rangle$. 
The partial order ${\mathbb X}(M)$ consists of two elements: the R-class of $1_M$, which we again denote by $1_M$, and the common R-class of 
all $b\in B$, which we denote by $\bf b$.
Clearly ${\bf b}\leq_{{\mathbb X}(M)} 1_M$. Thus, ${\mathbb Y}(M)$ has a unique maximal element $\{ {\bf b}, 1_M\}$, which we denote by $\bf y$. Note that for each $b\in B$ 
\[
b({\bf y}) = \{ {\bf b}\}.
\]
Two important to us conclusions follow from this equality. First, computing in $\langle {\mathbb Y}(M)\rangle$, for $b_0, \dots, b_n\in B$, we get 
\begin{equation}\label{E:hro}
b_0({\bf y})\vee \cdots \vee b_n({\bf y}) = \{ {\bf b}\}, 
\end{equation}
that is, the product $b_0({\bf y})\vee \cdots \vee b_n({\bf y})$ does not depend on $b_0, \dots, b_n$. Second, in the point based function array $\langle {\mathbb Y}(M)\rangle_{\bf y}$ over $M$, for $b_0, b_1\in B$,
we have 
\begin{equation}\label{E:hro2}
b_0(\bullet) = b_0({\bf y}) = b_1({\bf y}) = b_1(\bullet),  
\end{equation}
that is, all elements of $B$ are conjugate to each other (with the notion of conjugate as in Corollary~\ref{C:abst}). 

Let $I$ be the subset of $\langle (X_n)\rangle$ 
consisting of all words that contain a distinguished element of some $X_n$. The set $I$ is a two-sided ideal 
and that $\langle (X_n)\rangle$ is $I$-directed. 
By Corollary~\ref{C:yet}(ii), there exists a homomorphism 
\[
(f,g)\colon \langle {\mathbb Y}(M)\rangle_{\mathbf y} \to \gamma ( \langle (X_n(B))\rangle(\alpha))
\] 
and an ultrafilter 
${\mathcal V} \in \gamma( \langle (X_n(B))\rangle)$ such that 
\[
I\in f(\bullet),\, E\in {\mathcal V}\;\hbox{ and }\; f(\bullet) * {\mathcal V} = f(\bullet).
\]

Let $F$ consist of one point $\{ {\bf b}\}\in \langle {\mathbb Y}(M)\rangle$. So $F$ is a finite subset of $\langle {\mathbb Y}(M)\rangle$. Given a finite coloring of 
$\langle {\mathbb Y}(M)\rangle$, Corollary~\ref{C:abst} and the existence of the homomorphism $(f,g)$ above imply 
that there exists a basic sequence $(v_i, v_i')$ in $\langle (X_n(B))\rangle$ with $v_i'\in E$ that is conjugate $F$-$\langle {\mathbb Y}(M)\rangle$-tame. In view of the definition of $F$, 
\eqref{E:hro}, and \eqref{E:hro2} the coloring is fixed on sequences of the form \eqref{E:refix}. 
\end{proof}

Let $(X_n)$ be a sequence of pointed $M$-sets for a finite monoid $M$. We say that $(X_n)$ has the {\bf Ramsey property} if
for each finite coloring of $\langle (X_n)\rangle$ there exists a basic sequence $(w_i)$ in $\langle (X_n)\rangle$ such that
\begin{enumerate}
\item[---] each $w_i$ contains the distinguished element of $X_n$ as an entry;

\item[---] all words of the form
\[
a_0(w_{i_0})\cdots a_l(w_{i_l}),
\]
where $l\in {\mathbb N}$, $a_i\in M$ with at least one $a_i= 1_M$, are assigned the same color.
\end{enumerate}

A monoid $M$ is called {\bf Ramsey} if each sequence of pointed $M$-sets has the Ramsey property.

We deduce from Corollary~\ref{C:genseq} the following result characterizing Ramsey monoids. 

\begin{corollary}\label{C:moral}
\begin{enumerate}
\item[(i)] If $M$ is almost R-trivial and the partial order ${\mathbb X}(M)$ is linear, then $M$ is Ramsey.

\item[(ii)] If ${\mathbb X}(M)$ is not linear, then the sequence of pointed $M$-sets $X_n={\mathbb X}(M)$, with the canonical action of $M$
and with the R-class of $1_M$ as the distinguished point, does not have the Ramsey property.
\end{enumerate}
Thus, if $M$ is almost R-trivial, then $M$ is Ramsey if and only if the partial order ${\mathbb X}(M)$ is linear.
\end{corollary}

\begin{proof} (i) Fix a sequence of pointed $M$-sets $(X_n)$. We need to show that it has the Ramsey property.
One checks easily that linearity of ${\mathbb X}(M)$ implies that there exists an order preserving $M$-equivariant
embedding of ${\mathbb X}(M)$ to ${\mathbb Y}(M)$ mapping the top element of ${\mathbb X}(M)$ to a maximal element 
of ${\mathbb Y}(M)$---map
the R-class of $a$ to the set of all predecessors of the class of $a$ in ${\mathbb X}(M)$. We identify ${\mathbb X}(M)$ 
with its image in ${\mathbb Y}(M)$. 
Note that, by linearity of ${\mathbb X}(M)$,  ${\mathbb X}(M) = \langle {\mathbb X}(M)\rangle$,  
so ${\mathbb X}(M)$ is a subsemigroup of $\langle {\mathbb Y}(M)\rangle$.
Let $y_0 $ be the top element of ${\mathbb X}(M)$, which is 
the R-class of $[1_M]$. Since $[a]\vee [1_M] = [1_M]\vee [a] = [1_M]$, for the R-class $[a]$ of each $a\in M$, it follows 
immediately from Corollary~\ref{C:genseq} that $(X_n)$ has the Ramsey property. 
Since $(X_n)$ was arbitrary, we get the conclusion of (i).

(ii) Let $X_n$, $n\in {\mathbb N}$, be the pointed $M$-sets described in the statement of (ii).
Let $a,b\in M$ be two elements whose R-classes $[a]$ and $[b]$ are incomparable in ${\mathbb X}(M)$. Then $a\not\in bM$ and $b\not\in aM$,
which implies that
\begin{equation}\label{E:col}
[a]\not\in b{\mathbb X}(M)\;\hbox{ and }\;[b]\not\in a{\mathbb X}(M).
\end{equation}
We color $w\in \langle (X_n)\rangle$ with color $0$ if $[a]$
occurs in $w$ and its first occurrence precedes all the occurrences of $[b]$, if there are any. Otherwise, we color $w$ with color $1$. Let
$(w_i)$ be a basic sequence in $\langle (X_n)\rangle$ with the R-class $[1_M]$ of $1_M$ occurring in
each $w_i$. Then, in $a(w_0) w_1$,
$[a]$ occurs in $a(w_0)$ and, by \eqref{E:col}, $[b]$ does not occur in $a(w_0)$. It follows that $a(w_0) w_1$ is assigned color $0$. 
For similar reasons, $b(w_0) w_1$ is assigned color $1$. Thus, the Ramsey property fails for $(X_n)$.
\end{proof}

Using the characterization from Corollary~\ref{C:moral}, a more concrete characterization of Ramsey monoids among almost 
R-trivial monoids was recently given in \cite{KP}.
It turns out that almost R-trivial monoids are rarely Ramsey. Note, however, that, by Corollary~\ref{C:genseq}, 
one obtains Ramsey theorems even from non-Ramsey monoids; 
see the Furstenberg--Katznelson theorem in Section~\ref{Su:conc}.

\subsection{Some concrete applications}\label{Su:conc} 

We will present detailed arguments for showing that 
\begin{enumerate}
\item[(1)] the Furstenberg--Katzenlson theorem for located words and 

\item[(2)] the Hales--Jewett theorem for located left variable words 
\end{enumerate} 
follow from combining Corollaries~\ref{C:abstsup} and \ref{C:abst} with Corollary~\ref{C:yet}, in case of (1) through 
Corollary~\ref{C:genseq}. We also indicate how (3) Gowers' theorem and (4) Lupini's theorem are special cases of Corollary~\ref{C:genseq}.
We provide more details in the case (1) and (2) 
since (1) appears to be new and the derivation of (2) appears to be the most subtle one involving an application of Corollary~\ref{C:abst}. We emphasize however that all these derivations essentially amount to 
translations and specifications of our general results.

{\bf 1.} {\bf Furstenberg--Katznelson's theorem for located words.} We state here the Furstenberg--Katznelson theorem 
for located words. The original version from \cite{FK} is stated in terms of words and, as explained in the beginning of Section~\ref{Su:ramo}, is implied by our statement. We refer the reader 
to \cite{FK} for the original version. What follows in this point is, in essence, a translation of a particular case of Corollary~\ref{C:genseq} to the language used to state the Furstenberg--Katznelson theorem.  

We fix an element $x$, called a variable. For a set $C$ with $x\not\in C$, let 
\begin{equation}\label{E:exes}
X_n(C) = \{ n\}\times (C\cup \{ x\}).
\end{equation}

Fix now two finite disjoint sets $A, B$ with $x\not\in A\cup B$. If $w\in \langle (X_n(A\cup B))\rangle$, $x$ occurs in $w$, and $c\in A\cup B\cup \{ x\}$, then 
\begin{equation}\label{E:subb}
w[c] 
\end{equation}
is an element of $\langle X_n(A\cup B)\rangle$ obtained from $w$ by replacing each occurrence of $x$ by $c$. 

A {\bf reduced string in} $A$ is a sequence $a_0\cdots a_k$, possibly empty, such that $a_i\not= a_{i+1}$ for all $i<k$. 
With a sequence $c_0\cdots c_k$ with entries in $A\cup B$, we associate a reduced string $\overline{c_0\cdots c_k}$ in $A$ as follows. 
We delete all entries coming from $B$ thereby forming a sequence $c'_0\cdots c'_{k'}$ for some $k'\leq k$. In this sequence, we replace each run of each element of $A$  
by a single occurrence of that element forming a sequence $c_0''\cdots c''_{k''}$ with $k''\leq k'$. This sequence is $\overline{c_0\cdots c_k}$.  

Here is the statement of the Furstenberg--Katznelson theorem for located words. 

\smallskip

\noindent{\em Let $F$ be a finite set of reduced strings in $A$. 
Color $\langle X_n( A\cup B) \rangle$ with finitely many colors. There exists a basic sequence $(w_i)$ in $\langle X_n(B)\rangle$ 
such that $x$ occurs in each $w_i$ and, for each $n_0<\cdots < n_k$ and $c_0, \dots , c_k\in A\cup B$, the color of 
\[
w_{n_0}[c_0]\cdots w_{n_k}[c_k]
\] 
depends only on $\overline{c_0\cdots c_k}$ provided $\overline{c_0\cdots c_k}\in F$. }
\smallskip

This theorem is obtained by considering the monoid $J(A,B)$ from Section~\ref{Su:exa}. For brevity's sake, set 
\[
M = J(A,B).
\]
Forgetting about the Ramsey statement for a moment, we make some computations in ${\mathbb Y}(M)$. 

Observe that all elements of $B$ are in the same R-class, which we denote by $\bf b$, 
the R-class of each element of $A$ consists only of this element only, and the R-class of $1_M$ consists only of $1_M$. So, with some abuse of notation, 
we can write 
\[
{\mathbb X}(M) = \{ {\bf b}, 1_M\}\cup A.
\]
We have that, for each $a\in A$, 
\[
{\bf b}\leq_{{\mathbb X}(M)} a\leq_{{\mathbb X}(M)} 1_M
\]
and elements of $A$ are incomparable with each other with respect to $\leq_{{\mathbb X}(M)}$. 
The action of $M$ on ${\mathbb X}(M)$ is induced by the action of $M$ on itself by left multiplication. 

Pick $a_0\in A$. Note that the sets 
\[
\{ {\bf b}\}, \{ {\bf b}, a\}, \hbox{ for }a\in A, \hbox{ and }\{ {\bf b}, a_0, 1_M\} 
\]
are in ${\mathbb Y}(M)$, and we write $\bf b$, $a$, $1_0$ for these elements, respectively.
We notice that 
\begin{equation}\label{E:bar} 
{\bf b}\leq_{{\mathbb Y}(M)} a,\hbox{ for all } a\in A,\;\hbox{ and }\; {\bf b}, a_0\leq_{{\mathbb Y}(M)} 1_0, 
\end{equation}
and $\leq_{{\mathbb Y}(M)}$ does not relate any other two of the above 
elements. 
Furthermore, $1_0$ is a maximal element of ${\mathbb Y}(M)$. 
The action of $M$ on these elements is induced by the left multiplication action of $M$ on itself, so 
\[
a(1_0) =a \hbox{ and } b(1_0) = {\bf b},\hbox{ for }a\in A, b\in B. 
\]

Using relations \eqref{E:bar}, we observe that, for $c_0, \dots, c_k\in \{{\bf b}\}\cup A$, the product 
\[
c_0\vee\cdots \vee c_k
\]
in the semigroup of $\langle {\mathbb Y}(M)\rangle$ is equal to $\bf b$ if $c_i = {\bf b}$, for each $i\leq k$, 
or is obtained from $c_0\vee\cdots \vee c_k$ 
by removing all occurrences of $\bf b$ and shortening a run of each $a\in A$ to one occurrence of $a$, if 
$c_i\in A$, for some $i\leq k$. Thus, the map assigning to a sequence $c_0 \cdots c_k$ of elements of $A\cup B$ 
the element $c_0\vee\cdots \vee c_k$ of ${\mathbb Y}(M)$ factors through the map 
$c_0 \cdots c_k\to \overline{c_0 \cdots c_k}$ giving an injective map $\overline{c_0 \cdots c_k}\to c_0\vee\cdots \vee c_k$.

We note that $M$ acts on $X_n(A\cup B)$ as follows. We identify $X_n(A\cup B)$ with $M$ by identifying $(n,x)$ with $1_M$ and $(n,c)$ with $c$ for 
$c\in A\cup B$. Since $M$ acts on $M$ by left multiplication, this identification gives an action of $M$ on $X_n(A\cup B)$. We make 
$(n,x)$ the distinguished element, thereby turning $X_n(A\cup B)$ into a pointed $M$-set.  

We apply Corollary~\ref{C:genseq} to this sequence of pointed $M$-sets. In the statement of the theorem, we take ${\bf y}=1_0$ and, for the finite subset $\langle{\mathbb Y}(M)\rangle$, we take 
\[
\{ c_0\vee\cdots \vee c_k\colon c_0 \cdots c_k \in F\}. 
\]
Now, an application of Corollary~\ref{C:genseq} gives a basic sequence $(w_i')$ in 
$\langle (X_n(A\cup B))\rangle$. Let $w_i\in \langle X_n(B)\rangle$ be gotten from $w_i'$ by replacing each 
value taken in $A$ by $x$. By the discussion above, the sequence $(w_i)$ is as required.

{\bf 2.} {\bf The Hales--Jewett theorem for left-variable words.} 
The Hales--Jewett theorem for located words is just the Furstenberg--Katznelson theorem for located words with $A=\emptyset$.  
We state now and prove the Hales--Jewett theorem for located left-variable words as in \cite[Theorem~2.37]{To}. We use the notation 
as in \eqref{E:exes} and \eqref{E:subb}. We call $w\in \langle (X_n(C))\rangle$ {\bf left-variable} if the first entry in the sequence $w$ is of the form $(n,x)$ for some $n\in {\mathbb N}$.

\smallskip

\noindent {\em Let $B$ be a finite set. For each finite coloring of $ \langle (X_n(B))\rangle$, there exists a basic sequence $(w_i)$ in $\langle (X_n(B))\rangle$ such that $x$ 
does not occur in $w_0$, each $w_i$ with $i\geq 1$ is left-variable, and the color of the sequences 
\begin{equation}\label{E:fixat}
w_0w_{n_0}[b_0]\cdots w_{n_k}[b_k],
\end{equation}
with $0<n_0<\cdots < n_k$ and $b_0, \dots , b_k\in B$, is fixed. 
}

\smallskip

We will use the monoid $M=J(\emptyset, B)$ and apply Corollary~\ref{C:genj}.
Note that $M =B\cup \{ 1_M\}$ is in a bijective correspondence  
with $X_n(B)$ mapping each $b\in B$ to $(n,b)$ and $1_M$ to $(n, x)$. We transfer the left multiplication action of $M$ to $X_n(B)$ and make $(n,x)$ the distinguished 
element of $X_n(B)$. Thus, $X_n(B)$ is a pointed $M$-set. We consider the semigroup $\langle (X_n(B))\rangle$ with the induced endomorphism action of $M$. Let $E\subseteq \langle (X_n(B))\rangle$
consist of all left-variable elements. Note that $E$ is a right ideal in $\langle (X_n(B))\rangle$ and $\langle (X_n(B))\rangle$ is $E$-directed. Now Corollary~\ref{C:genj} produces 
a basic double sequence $(v_i, v_i')$ in $\langle (X_n)\rangle$, with $x$ occurring in each $v_i$ and with $v_i'\in E$ for each $i$. 
Fix $b\in B$, and let 
\[
w_0=v_{0}[b]\; \hbox{ and }\; w_i = v'_{i-1} v_{i} \hbox{ for }i\geq 1. 
\]
Clearly $x$ does not occur in $w_0$ and each $w_i$ is left-variable for $i\geq 1$. The coloring is fixed on sequences of the form \eqref{E:fixat} since it is fixed 
on ones of the form \eqref{E:refix}.

{\bf 3.} {\bf Gowers' theorem.} The monoid $G_k$ is defined in Section~\ref{Su:exa}. 
Gowers' Ramsey theorem from \cite{Go}, see \cite[Theorem~2.22]{To}, is obtained by applying Corollary~\ref{C:genseq} 
to $X_n= G_k$ with the left multiplication action and with the distinguished element $1_{G_k}$. 
We note that ${\mathbb X}(G_k)$ is linear, and we apply Corollary~\ref{C:genseq} as in the proof of 
Corollary~\ref{C:moral}(i).

{\bf 4.} {\bf Lupini's theorem.} Lupini's Ramsey theorem from \cite{Lu} is an infinitary version of a Ramsey theorem found by Barto{\v s}ova and Kwiatkowska in \cite{BK}. To prove it we consider the monoid $I_k$ defined in Section~\ref{Su:exa}. We take 
for $X_n = \{ 0, \dots, k-1\}$ with the natural action of $I_k$ and the distinguished element $k-1$. 
The result is obtained by applying Lemma~\ref{L:aux} and Theorem~\ref{T:abst}. We expand on this theme in Section~\ref{S:none}.

\subsection{The monoids $I_n$}\label{S:none}

We analyze here the monoids $I_n$, $n\in {\mathbb N}$, $n>0$, defined in Section~\ref{Su:exa}. 
As usual, we identify a natural number $n$ with the set $\{ 0, \dots, n-1\}$. 
The monoid $I_n$ is the monoid of all functions
$f\colon n\to n$ such that $f(0)=0$ and $f(i-1)\leq f(i)\leq f(i-1)+1$, for all $0<i<n$, taken with composition. We consider these monoids, on the one hand, 
to illustrate our notion of Ramsey monoids and, on the other hand, 
to answer a question of Lupini from \cite{Lu}. 
We will prove the following theorem.

\begin{theorem}\label{T:nor}
The monoids $I_n$, for $n\geq 4$, are not Ramsey. The monoids $I_1$, $I_2$, and $I_3$ are Ramsey.
\end{theorem}

We now state a theorem and question of Lupini \cite{Lu} in our terminology. For $k\in {\mathbb N}$,
let $w_k$ be a finite word in the alphabet $n=\{ 0, 1, \dots, n-1\}$ that contains an occurrence $n-1$.
Let $I_n(w_k)$  be equal to the set $\{ f(w_k)\colon f\in I_n\}$, where $f(w_k)$ is the word obtained from $w_k$ by applying $f$
letter by letter. We take $I_n(w_k)$ with the natural action of $I_n$ and with $w_k$ as the distinguished element.
Note that if $w_k$ is the word of length one whose unique letter is $n-1$, then $I_n(w_k)= n$ with the natural action of $I_n$ on $n$.

\begin{theorem*}[Lupini~\cite{Lu}]
Let $n>0$, and let $w_k=(n-1)$. Then the sequence of pointed $I_n$-sets $(I_n(w_k))_k$ has the Ramsey property.
\end{theorem*}

In \cite{Lu}, Lupini asked the following natural question:
does $(I_n(w_k))_k$ have the Ramsey property for every choice of words $w_k$, $k\in {\mathbb N}$?

The following corollary to Theorem~\ref{T:nor} answers this question in the negative.

\begin{corollary}\label{C:nega}
Let $n\geq 4$. For $k\in {\mathbb N}$, let $w_k = (01\cdots (n-1))$. Then the sequence of pointed $I_n$-sets $(I_n(w_k))_k$
does not have the Ramsey property.
\end{corollary}

\begin{proof} By Theorem~\ref{T:nor}, $I_n$ is not Ramsey for $n\geq 4$. It follows, by Corollary~\ref{C:moral} and by R-trivialiity of
$I_n$, that the sequence $X_k = I_n$, $k\in {\mathbb N}$, does not have the Ramsey property, where $I_n$ is considered as
a pointed $I_n$-set with the left multiplication action and with $1$ as the
distinguished element. Note that $I_n$ is isomorphic as a pointed $I_n$-set with $I_n(01\cdots (n-1))$ as witnessed by the function
\[
I_n\ni f\to f(01\cdots (n-1))\in I_n(01\cdots (n-1)).
\]
Thus, since $w_k=01\cdots (n-1)$, the sequence $(I_n(w_k))_k$ does not have the Ramsey property.
\end{proof}

We will give a recursive presentation of the monoid $I_n$ that may be of some independent interest and usefulness for future
applications. It will certainly make it easier for us to manipulate symbolically elements of $I_n$ below.
In the recursion, we will start with a trivial monoid and adjoin a tetris
operation as in \cite{Go} at each step of the recursion.

First, we present a general extension operation that can be applied to certain monoids equipped with an endomorphism and a distinguished element. 
Let $M$ be a monoid, let $f\colon M\to M$ be an endomorphism, and let $t\in M$ be such that for all $s\in M$ we have
\begin{equation}\label{E:com}
st=tf(s).
\end{equation}
Define
\[
\mu(M, t, f)
\]
to be the triple
\[
(N, \tau, \phi),
\]
where $N$ is a monoid, $\tau$ is an element of $N$, and $\phi$ is an endomorphism of $N$ that are
obtained by the following procedure. Let $N$ be the disjoint union of $M$ and the set $\{ \tau s\colon s\in M\}$, where
$\tau$ is a new element and the expression $\tau s$ stands for the ordered pair $(\tau, s)$. For $s\in M$, we write
$\tau^0 s$ for $s$ and $\tau^1s$ for $\tau s$. Define a function
$\phi\colon N\to M\subseteq N$ by letting, for $s\in M$ and $e=0,1$,
\[
\phi(\tau^e s) = t^e f(s),
\]
where $t^e f(s)$ is a product computed in $M$.
Define multiplication on $N$ be letting, for $s_1, s_2\in M$, and $e_1, e_2= 0,1$,
\[
(\tau^{e_1}s_1) \cdot (\tau^{e_2} s_2) = \begin{cases}
\tau^{e_1}(s_1s_2), \text{ if $e_2=0$;}\\
\tau (\phi(\tau^{e_1}s_1)s_2), \text{ if $e_2=1$.}
\end{cases}
\]
where, on the right hand side, $s_1s_2$ and $\phi(\tau^{e_1}s_1)s_2$ are products computed in $M$.
We write $\tau$ for $\tau1_M$. Note that $s_1\cdot s_2 = s_1s_2$ for $s_1, s_2\in M$, $\tau\cdot s = \tau s$ for $s\in M$ and $\tau\cdot \tau = \tau t$.
We will omit writing $\cdot$ for multiplication in $N$.

The following lemma is proved by a straightforward computation.

\begin{lemma}\label{L:mann}
$N$ is a monoid, $\phi$ is an endomorphism of $N$, and, for all $\sigma\in N$, we have relation \eqref{E:com}, that is,
\[
\sigma\tau = \tau\phi(\sigma).
\]
\end{lemma}

Later, we will need the following technical lemma.

\begin{lemma}\label{L:tech}
For $\sigma \in N$ and $s\in M$, there exists $s'\in M$ such that $\tau s \sigma=\tau s s'$.
\end{lemma}

\begin{proof} If $\sigma \in M$, then we can let $s'=\sigma$. Otherwise, $\sigma = \tau s_0$ for some $s_0\in M$.
Note that
\[
\tau s\sigma= \tau s \tau s_0 = \tau \tau f(s) s_0 = \tau t f(s) s_0 = \tau s t s_0,
\]
and we can let $s'= ts_0$.
\end{proof}

By recursion, we define a sequence of monoids with distinguished elements and endomorphisms.
Let $M_1$ be the unique one element monoid, let $t_1$ be its unique element, and let $f_1$ be its unique endomorphism.
Assume we are given a monoid $M_k$ for some $k\geq 1$ with an endomorphism $f_k$ of $M_k$
and an element $t_k$ with \eqref{E:com}. Define
\[
(M_{k+1}, t_{k+1}, f_{k+1})  = \mu(M_k, t_k, f_k).
\]

\begin{proposition}\label{P:repr}
For each $k\in {\mathbb N}$, $k>0$, $M_k$ is isomorphic to $I_k$.
\end{proposition}

\begin{proof} One views $I_{k-1}$ as a submonoid of $I_k$, for $k>1$, identifying $I_{k-1}$ with its image under the isomorphic
embedding $I_{k-1}\ni s \to s' \in I_k$, where
\[
s'(i) = \begin{cases} 0, \text{ if $i=0$;}\\
s(i-1)+1, \text{ if $0<i<k$.}
\end{cases}
\]
One checks that $t_k\in I_k$ given by
\[
t_k(i) = \begin{cases}
0, \text{ if $i=0$;}\\
i-1, \text{ if $0<i<k$.}
\end{cases}
\]
and $f_k\colon I_k\to I_k$ given by
\[
f_k(s)(i) = \begin{cases} 0, \text{ if $i=0$;}\\
s(i-1)+1, \text{ if $0<i<k$.}
\end{cases}
\]
fulfill the recursive definition of $(M_k, t_k, f_k)$.
\end{proof}

Since, as proved in Section~\ref{Su:exa}, $I_n$ is R-trivial, the partial order ${\mathbb X}(I_n)$
can be identified with $I_n$. We will make this identification and write $\leq_{I_n}$ for $\leq_{{\mathbb X}(I_n)}$.
We have the following recursive formula for $\leq_{I_n}$. Obviously, $\leq_{I_1}$ is  the unique partial order on
the one element monoid.

\begin{proposition}\label{P:recur}
Let $t_{n+1}^{e_1} s_1, t_{n+1}^{e_2}s_2 \in I_{n+1}$ with $s_1, s_2\in I_n$ and $e_1, e_2=0,1$.
Then $t_{n+1}^{e_1} s_1\leq_{I_{n+1}} t_{n+1}^{e_2}s_2$ if and only if
\[
e_2\leq e_1 \,\hbox{ and }\, s_1\leq_{I_n} f_n^{e_1-e_2}(s_2).
\]
\end{proposition}

\begin{proof} By Proposition~\ref{P:repr}, we regard $(I_{n+1}, t_{n+1}, f_{n+1})$ as obtained from the triple
$(I_n, t_n, f_n)$ via operation $\mu$.  In particular, we regard $I_n$ as a submonoid of $I_{n+1}$. We
also have $f_{n+1}(s) = f_n(s)$, for $s\in I_n$.

($\Leftarrow$) If $e_0=e_1$, the implication is obvious. The remaining case is $e_2=0$ and $e_1=1$. In this case,
we have $s_1\leq_{I_n} f_n(s_2)$, that is, $s_1 = f_n(s_2)s'$ for some $s'\in I_n$. But then
\[
t_{n+1} s_1 = t_{n+1} f_n(s_2)s' = s_2 (t_{n+1} s'),
\]
and $t_{n+1} s_1 \leq_{I_{n+1}} s_2$ as required.

($\Rightarrow$) Note that it is impossible to have $s_1\leq_{I_{n+1}} t_{n+1} s_2$ for $s_1, s_2\in I_n$. Indeed, this inequality
would give $s_1= t_{n+1} s_2\sigma$ for some $\sigma\in I_{n+1}$, which would imply, by Lemma~\ref{L:tech}, that
$s_1= t_{n+1} s_2s'$ for some $s'\in I_n$. This is a contradiction since $s_2s'\in I_n$. Thus,
$t_{n+1}^{e_1} s_1\leq_{I_{n+1}} t_{n+1}^{e_2}s_2$ implies $e_2\leq e_1$.

If $e_1=e_2=0$, then we have $s_1\leq_{I_{n+1}} s_2$, which means $s_1 = s_2\sigma$ for some $\sigma\in I_{n+1}$. If
$\sigma\in I_n$, then $s_1\leq_{I_n}s_2$, as required. Otherwise, $\sigma = \tau s'$ for some $s'\in I_n$, which gives
\[
s_1 = s_2\tau s' = \tau (f_n(s_2) s'),
\]
which is impossible since $f_n(s_2) s'\in I_n$.

If $e_1=e_2=1$, then we have $t_{n+1} s_1\leq_{I_{n+1}} t_{n+1} s_2$,
which means $t_{n+1} s_1 = t_{n+1} s_2 \sigma$ for some $\sigma\in I_{n+1}$.
By Lemma~\ref{L:tech}, this equality implies $t_{n+1} s_1 = t_{n+1} (s_2 s')$ for some $s'\in I_n$. Since $s_2s'\in I_n$,
this equality gives $s_1=s_2 s'$, so $s_1\leq_{I_n}s_2$.

The last case to consider is $e_1=1$ and $e_2=0$, that is, $t_{n+1} s_1\leq_{I_{n+1}} s_2$. Then $t_{n+1} s_1 = s_2 \sigma$ for
some $\sigma \in I_{n+1}$. Note that $\sigma\not\in I_n$, so $\sigma = t_{n+1} s'$ for some $s'\in I_n$. But then we have
\[
t_{n+1} s_1 = s_2 t_{n+1} s' = t_{n+1} f_n(s_2) s',
\]
which implies $s_1 = f_n(s_2)s'$, that is, $s_1\leq_{I_n} f_n(s_2)$.
\end{proof}

\begin{proof}[Proof of Theorem~\ref{T:nor}]
It is easy to see from Proposition~\ref{P:recur} that the orders
$\leq_{I_1}$, $\leq_{I_2}$, $\leq_{I_3}$ are linear. So, by Corollary~\ref{C:moral}, $I_1$, $I_2$, and $I_3$ are Ramsey.

By Corollary~\ref{C:moral}, it remains to check that the partial order $(I_n, \leq_{I_n})$ is not linear for $n\geq 4$.
By Proposition~\ref{P:repr}, we regard $(I_{n+1}, t_{n+1}, f_{n+1})$ as obtained from the triple
$(I_n, t_n, f_n)$ via operation $\mu$.
It follows from Proposition~\ref{P:recur} that
$\leq_{I_{n+1}}$ restricted to $I_n$ is equal to $\leq_{I_n}$. Thus, it suffices to show that $(I_4, \leq_{I_4})$ is not linear.
Note that the image of $f_3$ is equal $I_2$ and $I_2$ has two elements. So there exists
$s_0\in I_3$ such that $f_3(s_0) \not= 1_{I_2}$. Thus, since $1_{I_2} = 1_{I_3}$, we get $f_3(s_0)<_{I_3} 1_{I_3}$.
It then follows from Proposition~\ref{P:recur} that $t_4$ and $s_0$ are not comparable with respect to $\leq_{I_4}$. Indeed,
\[
t_4 = t_4^11_{I_3}\hbox{ and } s_0 = t_4^0 s_0.
\]
Since $0<1$, we have $s_0\not\leq_{I_4} t_4$; since $1_{I_3}\not\leq_{I_3} f_3^{1-0}(s_0)$, we have $t_4\not\leq_{I_4} s_0$.
\end{proof}

The monoid $I_4$ is the first one among the monoids $I_n$, $n>0$,
that is not Ramsey. Using Propositions~\ref{P:repr} and \ref{P:recur},
one can compute $\leq_{I_4}$ as follows.
Since $I_3$ is linearly ordered, one can list the four elements of $I_3$ as $a_3\leq_{I_3}  a_2\leq_{I_3}  a_1\leq_{I_3} 1$.
Then $I_4$ is equal to the disjoint union $I_3\cup t_4 I_3$ and the order $\leq_4$ is the transitive closure of the relations
\[
\begin{split}
&a_3\leq_{I_4}  a_2\leq_{I_4}  a_1\leq_{I_4} 1;\\
&t_4a_3\leq_{I_4}  t_4a_2\leq_{I_4}  t_4a_1\leq_{I_4} t_4;\\
&t_4\leq_{I_4} a_1;\\
&t_4a_1\leq_{I_4} a_3.
\end{split}
\]
One can check by inspection that $M_1 = I_4\setminus \{ t_4\}$ and $M_2 = I_4\setminus \{ a_2, a_3\}$ are submonoids of $I_4$.
They are R-trivial as submonoids of an R-trivial monoid \cite{St}. One easily checks directly that $\leq_{M_1}$ and $\leq_{M_2}$
are linear, therefore, $M_1$ and $M_2$ are Ramsey by Corollary~\ref{C:moral}. Thus, $I_4$ is not itself Ramsey, but it is
the union of two Ramsey monoids.

\end{document}